\documentclass[10pt,a4paper]{article}
\usepackage{a4wide}
\usepackage{times,xcolor}
\usepackage{hyperref}
\usepackage{mathrsfs} 
\usepackage{amssymb} 
\usepackage{amsmath} 
\usepackage{amsthm} 
\usepackage{graphicx} 
\usepackage{paralist}

\newtheorem{theorem}{Theorem}[section]
\newtheorem{proposition}[theorem]{Proposition}
\newtheorem{lemma}[theorem]{Lemma}

\newcommand{\G}[2]{G_{#1,#2}}
\newcommand{\Exp}{\,\mathbb{E}}
\renewcommand{\Pr}{\,\mathbb{P}}
\newcommand{\eps}{\varepsilon}
\newcommand{\given}{ \; \big| \; }

\newcommand{\alphac}[1]{\alpha_c^{#1}}
\newcommand{\chic}[1]{\chi_c^{#1}}

\DeclareMathOperator{\Bin}{Bin}

\title{Bounded monochromatic components for random graphs}

\author{
Nicolas Broutin \\
Inria
\and
Ross J. Kang\thanks{This work was initiated while this author was at McGill University and supported by a NSERC Postdoctoral Fellowship. This author's research was also supported by a NWO Veni Grant.}\\
Radboud University Nijmegen
}

\begin{document}

\maketitle

\begin{abstract}
We consider vertex partitions of the binomial random graph $G_{n,p}$. For $np\to\infty$, we observe the following phenomenon: in any partition into asymptotically fewer than $\chi(G_{n,p})$ parts, i.e.~$o(np/\log np)$ parts, one part must induce a connected component of order at least roughly the average part size.

Stated another way, we consider the {\it $t$-component chromatic number}, the smallest number of colours needed in a colouring of the vertices for which no monochromatic component has more than $t$ vertices. As long as $np \to \infty$, there is a threshold for $t$ around $\Theta(p^{-1}\log np)$: if $t$ is smaller then the $t$-component chromatic number is nearly as large as the chromatic number, while if $t$ is greater then it is around $n/t$.

For $0 < p <1$ fixed, we obtain more precise information. We find something more subtle happens at the threshold $t = \Theta(\log n)$, and we determine that the asymptotic first-order behaviour is characterised by a non-smooth function. Moreover, we consider the {\it $t$-component stability number}, the maximum order of a vertex subset that induces a subgraph with maximum component order at most $t$, and show that it is concentrated in a constant length interval about an explicitly given formula, so long as $t = O(\log \log n)$.

We also consider a related Ramsey-type parameter and use bounds on the component stability number of $G_{n,1/2}$ to describe its basic asymptotic growth.


\end{abstract}

\section{Introduction}\label{sec:intro}

For $t$ a positive integer, the {\it $t$-component stability number} $\alphac{t}(G)$ of a graph $G$ is the maximum order of a {\it $t$-component set} --- a vertex subset that induces a subgraph with maximum component order at most $t$. The {\it $t$-component chromatic number} $\chic{t}(G)$ is the smallest number of colours needed in a {\it $t$-component colouring} --- a colouring of the vertices such that colour classes are $t$-component sets.  Note that $\chic{t}(G) \ge |V(G)|/\alphac{t}(G)$ for any graph $G$ and any positive integer $t$.

We study the $t$-component chromatic and stability numbers of $\G{n}{p}$, where $\G{n}{p}$ as usual denotes the Erd\H os--R\'enyi random graph with vertex set $[n] = \{1, \ldots, n \}$ and edges included independently at random with probability $p$, $0 < p < 1$.
We say that a property $A_n$ of $\G{n}{p}$ holds asymptotically almost surely (a.a.s.)~if $\Pr(A_n) \to 1$ as $n \to \infty$.
We use standard notational conventions: $q = 1 - p$ and $b = 1/q$.  
Unless specified otherwise, the base of logarithms is natural.

If $t = 1$, then $\chic{t}(\G{n}{p})$ coincides with the notion of the chromatic number $\chi(\G{n}{p})$ of $\G{n}{p}$, a parameter of intensive study in random graph theory. 
 For fixed $0 < p < 1$, Grimmett and McDiarmid~\cite{GrMc75} conjectured that $\chi(\G{n}{p}) \sim n/(2 \log_b n)$ a.a.s.  This remained a major open problem in random graph theory for over a decade, until Bollob\'as~\cite{Bol88} used martingale techniques to establish the conjecture; earlier, Matula~\cite{Mat87} had devised an independent method that was later proved to also confirm the conjecture~\cite{MaKu90}. {\L}uczak~\cite{Luc91a} used martingale concentration to extend Matula's method to sparse random graphs and showed that, for any fixed $\eps > 0$, there exists $d_0$ such that
\[\frac{(1 - \eps)np}{2 \log np} \le \chi(\G{n}{p}) \le \frac{(1 + \eps)np}{2 \log np}\] 
a.a.s.~if $np \ge d_0$. 
This reviews classic work in the area, but there has been tremendous further activity from many perspectives, cf.~e.g.~\cite{Coj13,CoVi13}; for further background on colouring random graphs,
see~\cite{Bol01,JLR00,KaMc15}. 

We begin with some basic observations about the $t$-component chromatic number.  Let $G$ be a graph and $t$ a positive integer.  Since a $t$-component set is a $(t + 1)$-component set, it follows that $\chic{t}(G) \ge \chic{t+1}(G)$. 
Also, each colour class of a $t$-component colouring can be properly coloured with at most $t$ colours, and it follows that $\chic{t}(G) \ge \chi(G)/t$.  
Moreover, any partition of the vertex set into $t$-sets is a $t$-component colouring.  We thus have the following range of values for $\chic{t}(G)$.
\begin{proposition}\label{prop:basic}
For any graph $G$ and positive integer $t$,
\[ \frac{\chi(G)}{t} \le \chic{t}(G) \le \min\left\{ \left\lceil \frac{|V(G)|}{t} \right\rceil, \chi(G) \right\}.\]
\end{proposition}

Roughly, we prove that $\chic{t}(\G{n}{p})$ is likely to be close to the upper end of the range implied by Proposition~\ref{prop:basic}: a.a.s.~it is close to $\chi(\G{n}{p})$ if $t(n) = o(\log_b np)$ and to $n/t$ if $t(n) = \omega(\log_b np)$. 
This has a compact qualitative interpretation: in any partition of the vertices of $\G{n}{p}$ into asymptotically fewer than $\chi(\G{n}{p})$ parts, one part must induce a subgraph having a large sub-component, about as large as the average part size.
This statement, made more precise in Theorem~\ref{thm:chi} below, concerns $\G{n}{p}$ with $np\to \infty$ as $n\to\infty$. For most of the paper however, we focus on the dense case, i.e.~with $p$ fixed between $0$ and $1$.

An interesting question is how to characterise $\chic{t}(\G{n}{p})$ at the threshold $t=\Theta(\log n)$. At this point, the two trivial upper bounds in Proposition~\ref{prop:basic} are of the same asymptotic order, and we see that something more subtle takes place.
Our main result is an explicit determination of $\chic{t}(\G{n}{p})$ assuming that $t/\log n$ is convergent as $n\to\infty$.
We find it convenient to set some notation:
given $\tau,\kappa>0$, define
\begin{align}
\iota(\tau,\kappa) =
\frac12\left(\left(\kappa-\tau\left\lfloor\frac{\kappa}{\tau}\right\rfloor\right) \left(\kappa-\tau\left\lfloor \frac{\kappa}{\tau}\right\rfloor-\tau\right) - \kappa(\kappa-\tau-2)\right).
\label{eqn:iota}
\end{align}

The following technical lemma is crucial; its proof can be found in the appendix. See also Figure~\ref{fig:chimediumplot}.

\begin{lemma}\label{lem:iota}
Let $\kappa=\kappa(\tau)$ be defined by the implicit equation $\iota(\tau,\kappa)=0$, for $\iota$ as defined in~\eqref{eqn:iota}. Then $\kappa : (0,\infty) \to \mathbb R$ is a well-defined function with the following properties.
\begin{enumerate}
\itemsep0pt
	\item\label{itm:iota,c} Over all of $(0,\infty)$, the function $\kappa$ is positive, increasing, piecewise convex, and continuous.
	\item\label{itm:iota,d} If $\tau \in (0,2]$, then $\kappa$ is close to $\tau+2$, with equality for $\tau=2/i$, $i\in \mathbb N$;
	otherwise $\kappa = \tau + \tau/(\tau-1)$.
	Moreover, $\tau+1< \kappa \le \tau +2$ always.
\end{enumerate}
\end{lemma}

We may now state our main result.

\begin{theorem}\label{thm:chi,medium}
Fix $0 < p < 1$.  Suppose $t = t(n) \sim \tau \log_b n$ as $n \to \infty$ for some $\tau > 0$ and let $\kappa=\kappa(\tau)$ be the unique positive real guaranteed by Lemma~\ref{lem:iota}.
Then a.a.s.
\begin{align*}
\chic{t}(\G{n}{p}) \sim \frac{n}{\kappa \log_b n}.
\end{align*}
\end{theorem}

Lemma~\ref{lem:iota} implies that $\kappa\to2$ as $\tau\downarrow 0$, and so we may view Theorem~\ref{thm:chi,medium} as a non-trivial extension of the aforementioned result of Bollob\'as~\cite{Bol88} on the chromatic number.
We shall see in Section~\ref{sec:exp} that the expected number of $t$-component $(\kappa\log_b n)$-sets is dominated by those with 
nearly all components of the maximum order $t$.
It is thus the remainder term, $\kappa\log_b n-t\lfloor (\kappa\log_b n)/t\rfloor$, that explains the non-smooth behaviour of $\kappa$ as a function of $\tau$.
Theorem~\ref{thm:chi,medium} follows from a first moment method, using a general asymptotic count of set partitions and an optimisation of the non-edge count, together with an involved second moment argument. 

\begin{figure}[ht]
\begin{center}
\includegraphics[height=7.4213333cm]{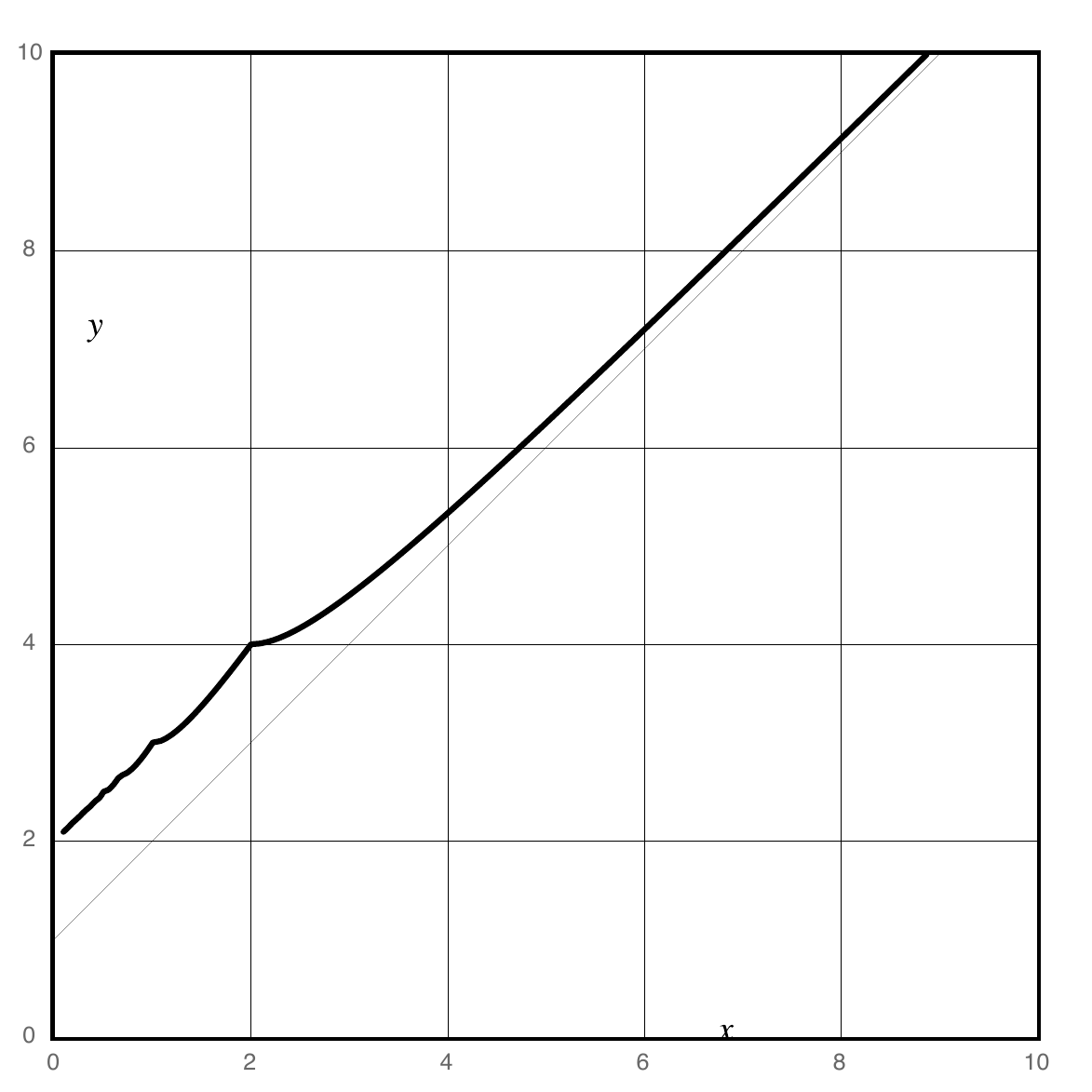} 
\includegraphics[height=7.4213333cm]{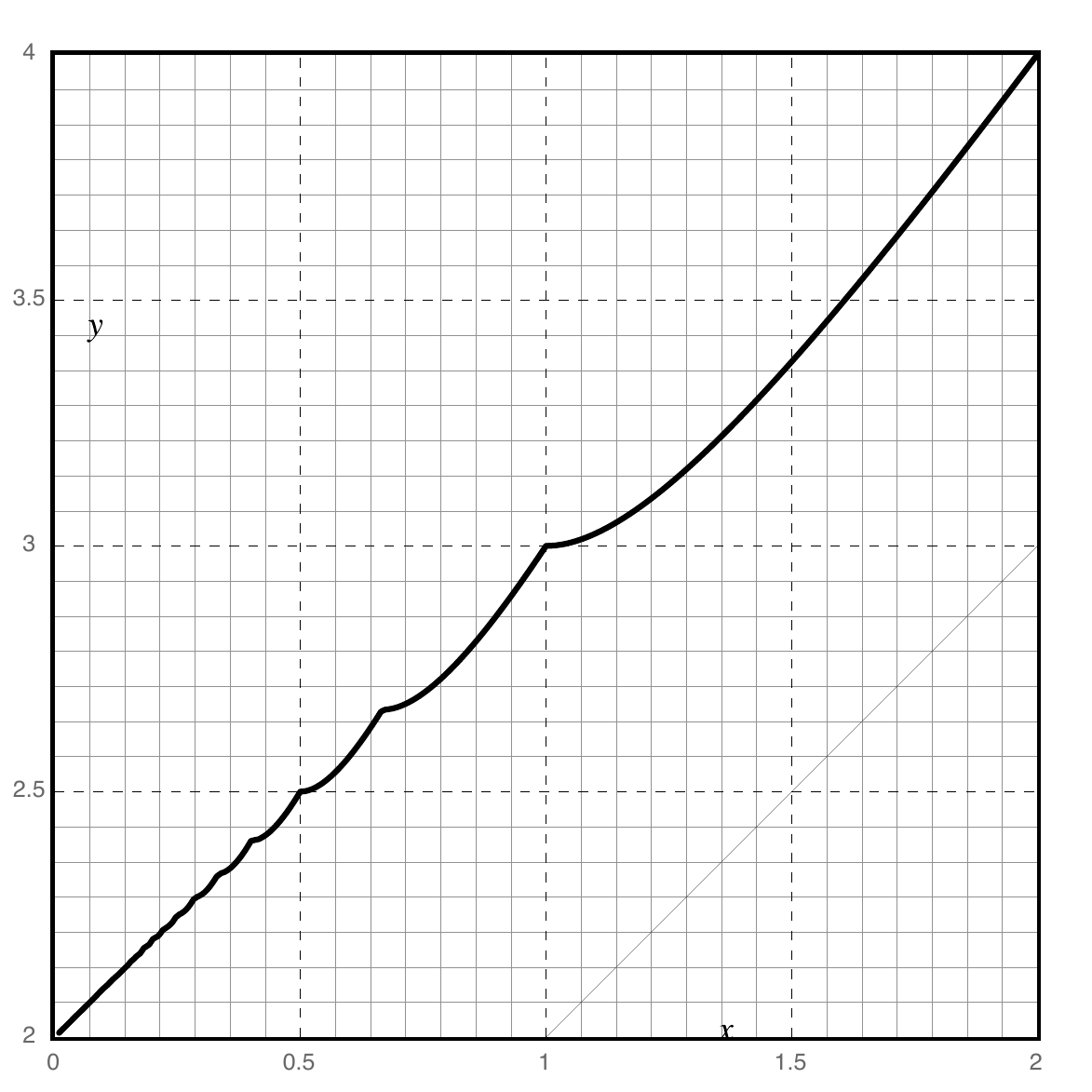} 
\end{center}
\caption{Plots of $\iota(x,y) = 0$, a function determining the behaviour in Theorem~\ref{thm:chi,medium}, and $y=x+1$.\label{fig:chimediumplot}}
\end{figure}

We also obtain an explicit, precise formulation for $\alphac{t}(\G{n}{p})$ when $t$ is bounded 
above by a slowly growing function of $n$. 
The formula in Theorem~\ref{thm:alpha} can be viewed as extending (up to the $\Theta(1)$ 
additive error term) the explicit formulation of the stability number $\alpha(\G{n}{p})$ 
of $\G{n}{p}$ obtained by Matula~\cite{Mat70, Mat72} (cf.~also Bollob\'as and 
Erd\H{o}s~\cite{BoEr76}).

\begin{theorem}\label{thm:alpha}
Fix $0 < p < 1$.  If $t=t(n) \le \log \log_b n$, then a.a.s.
\begin{align*}
\alphac{t}(\G{n}{p})
= 2\log_b n + t - 2\log_b t - \frac{2\log_b \log_b np}t + \Theta(1).
\end{align*}
\end{theorem}

The proof of this theorem is by way of bounds from enumerative combinatorics on the number 
of set partitions with bounded block size, and a second moment argument using a large deviations 
inequality. The condition $t(n) \le \log \log_b n$ marks roughly when specific set partition 
bounds are superseded by a generic bound, and our lower and upper estimates on the first moment diverge. We wonder how sharp this condition is with respect to constant-width 
concentration of $\alphac{t}(\G{n}{p})$. Such concentration 
is impossible when $t=\Omega(\sqrt{\log n})$, due to a term in the first 
moment that fluctuates unpredictably based on the value of 
$k/t-\lfloor k/t\rfloor$. (This rounding term has less impact when $t$ and $k$ have the same asymptotic order.)

Incidental to our sharp determination of the component stability number in 
Theorem~\ref{thm:alpha}, we obtain a good estimate of the component chromatic number 
for $t(n) \le \log \log_b n$. This is a small modification of Theorem~\ref{thm:alpha} 
for stronger concentration with slightly smaller sets, and then a close adaptation of the 
arguments in Section~5 of~\cite{FKM10} or in earlier work~\cite{McD90}. 
This adaptation is left to the reader.

\begin{theorem}\label{thm:chi,small}
Fix $0 < p < 1$.  If $t=t(n) \le \log \log_b n$, then a.a.s.
\begin{align*}
\chic{t}(\G{n}{p}) 
= \frac{n}{2\log_b n + t - 2\log_b t - \frac{2\log_b \log_b np}t + \Theta(1)}.
\end{align*}
\end{theorem}

Last, in a simpler application of our determination of the component stability number, we introduce a related Ramsey-type parameter and find its basic asymptotic behaviour.
Recall that the (diagonal, two-colour) Ramsey number is the smallest integer $R(k)$ for which any graph on $R(k)$ vertices contains a set of $k$ vertices that induces either a stable set or a clique as a subgraph.
The development of bounds for $R(k)$ as $k\to\infty$ is an important and difficult area of mathematics with over eight decades of history~\cite{Erd47,ErSz35}.
We now consider a generalisation of $R(k)$ where the notion of $t$-component set replaces that of stable set.
The {\em $t$-component Ramsey number} is the smallest integer $R^t(k)$ for which any graph on $R^t(k)$ vertices must contain a set of at least $k$ vertices that is a $t$-component set in either the graph or its complement.
We treat $t$ as a function of $k$.
Clearly, the $1$-component Ramsey number $R^1(k)$ coincides with $R(k)$, and by classic arguments (that use bounds on $\alpha(\G{n}{1/2})$)~\cite{Erd47,Spe77} has exponential growth in $k$. At the other extreme, $R^k(k)$ is trivially $k$. So we expect to see a dramatic decrease in $R^t(k)$ by increasing $t$ from $1$ to $k$. Note also that $R^t(k)$ is non-increasing in $t$.
The next result uses bounds on $\alphac{t}(\G{n}{1/2})$ and shows that $R^t(k)$ is at least exponential in $k$ in nearly the entire range of $t$, i.e.~the change from exponential to polynomial growth occurs in a narrow interval near $t = k$.

\begin{proposition}\label{prop:ramseylower}
Fix $0\le \eps < 1/2$. Then, as $k\to\infty$,
\begin{align*}
R^{\lfloor(1-\eps)k\rfloor}(k) \ge (1+o(1)) \frac{k}{3e}2^{\eps(1-\eps)k}.
\end{align*}
\end{proposition}

As we discuss in Section~\ref{sec:ramsey}, this result can be complemented by a K\H{o}v\'ari--S\'os--Tur\'an-type result.

\paragraph{Further remarks:}
\begin{itemize}
\itemsep0pt
\item[$\star$] Both the $t$-component chromatic number~\cite{ADOV03,BeSz07,BeSz09,EdFa05,EsJo14,EsOc14+,HST03,
Kaw09,KaMo07,KMRV97,LMST08,LiOu15+,MaPr08} and the $t$-component stability 
number~\cite{EdFa01,EdFa08,HPT08,JaTh08,Rah14+} have been actively considered from several viewpoints, 
especially in graph theory and theoretical computer science.
\item[$\star$] Note that $\alphac{t}(G)$ has often been studied in the following form: 
given $G$ and 
$t$, the {\em $t$-fragmentability} of $G$ is essentially $(|V(G)| - \alphac{t}(G))/|V(G)|$. 
This for instance has been considered in sparse random graphs as a watermark for feasibility 
of vaccination protocols in networks~\cite{BJM07,JaTh08}.
\item[$\star$] It is worth mentioning related work (involving the second author), where instead of 
component order we bound the (average) degree~\cite{FKM10,FKM14,KaMc10}.
Macroscopically, these parameters exhibited a similar threshold. However, the behaviour at the threshold was smooth and the magnitude of the threshold was of a different order in sparse random graphs. In Section~\ref{sec:sparse} we discuss this latter difference.
\item[$\star$] When $t$ is fixed, the property of being a $t$-component set is a hereditary 
property ---that is, it is a graph property that is closed under vertex-deletion--- whereupon 
 broad results on hereditary colourings apply~\cite{BoTh95,BoTh00,Sch92}. 
However, it is important here that we allow $t$ to grow as a function of~$n$.
\item[$\star$] Bounded monochromatic components of random graphs are also considered in the 
separate context of partitions of the edge set~\cite{BFKLS11,SST10}, a problem related to 
Achlioptas processes that control the growth of several ``giants'' simultaneously.
\end{itemize}

\paragraph{Plan of the paper.}
In Section~\ref{sec:exp}, we conduct an analysis of the expected number of 
$t$-component $k$-sets in $\G{n}{p}$, mainly via asymptotic set partition and non-edge counts.
We prove Theorem~\ref{thm:chi,medium} in Section~\ref{sec:chi} with a three-part second moment 
argument. In Section~\ref{sec:alpha}, we use an easier second moment argument that 
applies a large deviations inequality in order to prove Theorem~\ref{thm:alpha}. In Section~\ref{sec:sparse}, we discuss results for random graphs with smaller edge density. In Section~\ref{sec:ramsey}, we study the Ramsey-type problem.

\section{The expected number of $t$-component $k$-sets}\label{sec:exp}

Let $\mathcal{S}_{n,t,k}$ be the collection of $t$-component $k$-sets in $\G{n}{p}$.
This section is devoted to analysing the expected behaviour of $|\mathcal{S}_{n,t,k}|$: 
this governs the asymptotic behaviour of $\chic{t}(\G{n}{p})$.
We divide our analysis into lower and upper bounds on $\Exp(|\mathcal{S}_{n,t,k}|)$, 
partly because these bounds have different scopes.
These bounds depend mostly on sharp non-edge counts, and asymptotic estimates on 
the number of set partitions with bounded block size. We often analyse set partitions 
with the help of some analytic combinatorics. An important remark is that our expectation estimates 
naturally divide with respect to the value of $k/t$, either less than or greater than $2$, as in 
the former case the count of set partitions is much simpler.

Understanding the expectation computations may provide some insight into the formulas in 
Theorems~\ref{thm:chi,medium} and~\ref{thm:alpha}. For those readers who prefer to skip or 
skim over the rest of this section, the main results we require later in are the 
following two propositions and Lemma~\ref{lem:careful}.

\begin{proposition}[First-order estimate for $t = \Theta(\log n)$]\label{prop:exp,mediumt}
Suppose $0 < p < 1$ is fixed and $\eps>0$ is a small enough constant.
Suppose $t=t(n) \sim \tau \log_b n$ as $n \to \infty$ for some $\tau > 0$ and let $\kappa$ be the unique positive real satisfying $\iota(\tau,\kappa)=0$, for $\iota$ as defined in~\eqref{eqn:iota}.
\begin{enumerate}
\itemsep0pt
\item\label{prop:exp,mediumt,upper,part}
If
$k = k(n) \sim (\kappa+\eps) \log_b n$ as $n \to \infty$,
then $\Exp(|\mathcal{S}_{n,t,k}|) \le \exp((1+o(1))\iota(\tau,\kappa+\eps)(\log n)^2/\log b)$. 
\item\label{prop:exp,mediumt,lower,part}
If
$k = k(n) \sim (\kappa-\eps) \log_b n$ as $n \to \infty$,
then $\Exp(|\mathcal{S}_{n,t,k}|) \ge \exp((1+o(1))\iota(\tau,\kappa-\eps)(\log n)^2/\log b)$. 
\end{enumerate}
\end{proposition}

\begin{proposition}[Constant-width estimate for $t \le \log \log_b np$]\label{prop:exp,smallt}
Fix $0 < p < 1$. Suppose $t=t(n)$ satisfies $t\le \log \log_b np$.
\begin{enumerate}
\itemsep0pt
\item\label{prop:exp,smallt,upper,part}
If $k = k(n)$ satisfies as $n \to \infty$ that 
\begin{align*}
k \ge 2\log_b n + t -2\log_b t- \frac{2\log_b \log_b np}t + \frac{10}{\log b},
\end{align*}
then $\Exp(|\mathcal{S}_{n,t,k}|) \le \exp(-k)$ for $n$ large enough.
\item\label{prop:exp,smallt,lower,part}
If $k = k(n)$ satisfies as $n \to \infty$ that $k \ge \log_b n$ and 
\begin{align*}
k \le 2\log_b n + t - 2\log_b t - \frac{2\log_b \log_b np}{t} -\frac{2}{\log b},
\end{align*}
then $\Exp(|\mathcal{S}_{n,t,k}|) \ge \exp(k)$ for $n$ large enough.
\end{enumerate}
\end{proposition}

We use Proposition~\ref{prop:exp,mediumt} in Section~\ref{sec:chi} 
for the $t=\Theta(\log n)$ regime, and Proposition~\ref{prop:exp,smallt} in Section~\ref{sec:alpha} 
for the proof of Theorem~\ref{thm:alpha}.
Proposition~\ref{prop:exp,mediumt} follows from Propositions~\ref{prop:exp,mediumt,upper,1}, 
\ref{prop:exp,mediumt,upper,2}, \ref{prop:exp,mediumt,lower,1}, and~\ref{prop:exp,mediumt,lower,2}. 
Proposition~\ref{prop:exp,smallt} follows from Lemmas~\ref{lem:exp,smallt,upper,1} 
and~\ref{lem:exp,smallt,lower}. 

The following calculations will be useful when dealing with bounds involving 
$\iota$ as defined in~\eqref{eqn:iota}. The proof is found in the appendix.

\begin{lemma}\label{lem:careful}
For $\tau>0$, let $\kappa$ be the unique positive real satisfying $\iota(\tau,\kappa)=0$, for $\iota$ as defined in~\eqref{eqn:iota}.\begin{enumerate}
\itemsep0pt
\item\label{prop:careful,minus}
If $0 \le \eps < \tau\left(\left\lfloor\frac{\kappa}{\tau}\right\rfloor+1\right) - \kappa$, then $\iota(\tau,\kappa+\eps) = -\eps\left(\tau\left\lfloor \frac{\kappa}{\tau}\right\rfloor-1\right) < -\eps$.
\item\label{prop:careful,divide}
If $\tau | \kappa$ and $0 < \eps<\tau$, then $\iota(\tau,\kappa-\eps) = \eps$.
\item\label{prop:careful,notdivide}
If $\tau\not|\kappa$ and $0 \le \eps \le \kappa - \tau\left\lfloor\frac{\kappa}{\tau}\right\rfloor$, then $\iota(\tau,\kappa-\eps) = \eps\left(\tau\left\lfloor \frac{\kappa}{\tau}\right\rfloor-1\right) > \eps$.
\end{enumerate}
\end{lemma}

\subsection{Upper bounds on $\Exp(|\mathcal{S}_{n,t,k}|)$}

\begin{lemma}\label{lem:exp,larget,upper}
Suppose $p = p(n)$ satisfies $0 < p < 1$ and $n p \to \infty$ as $n \to \infty$.
Suppose $t=t(n)$ and $k = k(n)$ satisfy that $t,k \to \infty$ as $n \to \infty$.
Furthermore assume $t=O(\log_b np)$, $t \ge k/2$ (so that $1 \le k/t \le 2$) and 
\begin{align*}
k \ge t + \frac{k}{t}\log_b \frac{np}{pt + \log np} + \frac{6}{\log b}.
\end{align*}
Then $\Exp(|\mathcal{S}_{n,t,k}|) \le \exp(-t)$ for $n$ large enough.
\end{lemma}

\begin{proof}
We estimate the probability contribution of all $t$-component $k$-sets by classifying them according to partitions of $[k]$ so that there are no edges between any pair of parts.
Naturally, we could first consider the component structure as such a partition (ignoring what happens inside each component). However, we find it convenient to simplify our accounting by taking coarser partitions.
For a given $t$-component $k$-set, we group the connected components into possibly larger vertex subsets as follows.
We form a first such set $X_1$ by including just the largest component, unless it has at most $t/2$ vertices, in which case we add just the second largest component to the group, unless the resulting group has at most $t/2$ vertices, and so on. Then we form a second set $X_2$ in a similar way with the remaining components.
After this second grouping, all the remaining components (if there are any) are grouped into a third set $X_3$.
By construction, $t/2 \le |X_1|,|X_2| \le t$ and since $t\ge k/2$ we have that $|X_3| \le k-t\le t$.

From the above discussion, to upper bound the expectation of $|\mathcal{S}_{n,t,k}|$ it suffices to upper bound that of the number $k$-sets of $[n]$ that induce a partition of $[k]$ with part sizes $k_1$ (possibly $0$), $k_2$ and $k_3$ such that $0\le k_1\le k_2 \le k_3 \le t$, $k_1 \le k-t$ and $k_2 \ge t/2$, and with no edges between any two parts.
Since $k_3 = k-k_1-k_2$, the total number of non-edges between parts is expressed by
\begin{align*}
f(k_1,k_2)
:= k_1k_2 + (k_1+k_2)(k-k_1-k_2)
= k(k_1+k_2)-k_1^2-k_1k_2-k_2^2.
\end{align*}
In the following optimisation, we show that under the above constraints $f(k_1,k_2) \ge t(k-t)$ always.
For $k_1$ fixed with $0 \le k_1 \le k-t$, $f(k_1,k_2)$ is non-negative and concave in $k_2$ for $0 \le k_2 \le k-k_1$, and so minimised by evaluating at extreme values for $k_2$. The properties of the partition imply that $\max\{k_1,t/2,k-t-k_1\} \le k_2 \le (k-k_1)/2$.

Consider the three-term maximisation for the lower extreme of $k_2$.
Using $t\ge k/2$, observe that $k-3t/2 \le k/2-t/2\le t/2$.
Note $k_1\ge k-t-k_1$ is equivalent to $k_1 \ge k/2-t/2$, while $t/2\ge k-t-k_1$ is equivalent to $k_1 \ge k-3t/2$. These observations imply that the maximisation is attained by
\begin{enumerate}
\itemsep0pt
\item\label{itm:i}
$k-t-k_1$ if $k_1\le k-3t/2$,
\item\label{itm:ii}
$t/2$ if $k-3t/2 \le k_1 \le t/2$, and
\item\label{itm:iii}
$k_1$ if $k_1\ge t/2$.
\end{enumerate}

For case~\ref{itm:i}, $f(k_1,k-t-k_1) = t(k-t)+(k-t)k_1-k_1^2$ is concave in $k_1$ and so minimised over $0 \le k_1\le k-3t/2$ at $k_1 = 0$ or $k_1 = k-3t/2$.
In the former case we have $f(0,k-t) = t(k-t)$.
In the latter we get $f(k-3t/2,t/2) = t(k-t)+(2k-3t)/4$, which is at least $t(k-t)$ as long as $k-3t/2\ge 0$ (and otherwise case~\ref{itm:i} is vacuous).

For case~\ref{itm:ii}, $f(k_1,t/2) = (k-t/2)t/2+(k-t/2)k_1-k_1^2$ is concave in $k_1$ and so minimised over $\max\{k-3t/2,0\} \le k_1 \le t/2$ at $k_1 = \max\{k-3t/2,0\}$ or $k_1 = t/2$.
In the former case we already checked $f(k-3t/2,t/2) \ge t(k-t)$ as long as $k-3t/2\ge 0$; otherwise, we have $f(0,t/2)= (k-t/2)t/2$, which is at least $t(k-t)$ for $k-3t/2\le 0$.
In the latter case we get $f(t/2,t/2) = (k-3t/4)t > t(k-t)$.

For case~\ref{itm:iii}, $f(k_1,k_1) = 2kk_1-3k_1^2$ is concave in $k_1$ and so minimised over $t/2 \le k_1 \le k/3$ at $k_1 = t/2$ or $k_1 = k/3$. In the former case we already checked that $f(t/2,t/2) > t(k-t)$.
In the latter case we get $f(k/3,k/3) = k^2/3 > t(k-t)$.

For the upper extreme of $k_2$, we evaluate $f(k_1,(k-k_1)/2) = k^2/4+(k/2)k_1 -(3/4)k_1^2$.
This is concave in $k_1$ and so minimised over $0\le k_1 \le k/3$ when $k_1 = 0$ or $k_1 = k/3$.
In the former case we have $f(0,k/2) = k^2/4 \ge t(k-t)$.
In the latter case we get $f(k/3,k/3) > t(k-t)$.

This completes the optimisation to check that in all such partitions the total number $f(k_1,k_2)$ of non-edges between parts is at least $t(k-t)$.
As there are crudely at most $3^k$ such partitions of $[k]$, we obtain
\begin{align}\label{eqn:star}
\Exp(|\mathcal{S}_{n,t,k}|)
 \le \binom{n}{k} 3^k q^{t(k-t)}
 \le  \left(\frac{en}{k}\right)^k \cdot3^k q^{t(k-t)},
\end{align}
using $\binom{x}{y} \le (e x/y)^y$.
Taking the logarithm and dividing by $t$, we get for $n$ large enough that
\begin{align*}
\frac{\log \Exp(|\mathcal{S}_{n,t,k}|)}{t}
& \le \frac{k}{t} \log \left(\frac{en}{k}\right) + \frac{k}{t}\log 3 - (k-t)\log b\\
& \le \frac{k}{t} \log np - \frac{k}{t} \log pk +4.2 - (k-t)\log b,
\end{align*}
since $k/t \le 2$ and $\log 2/t \to 0$.
Now, the assumed lower bound on $k$ implies both that
\begin{align*}
(k-t)\log b \ge \frac{k}{t}\log np - \frac{k}{t}\log(pt + \log np) + 6
\end{align*}
and $pk \ge pt + \log np$. 
(The last inequality can be seen by first noting that $k \ge t+(1+o(1))\log_b np$, 
so that $k/t-1 = \Omega(1)$, and then applying the inequality again to obtain 
$pk \ge pt + (1+o(1))\frac{k}{t}\log \frac{np}{\log np} \ge pt + \log np$ for $n$ large enough.)
We then have $\log \Exp(|\mathcal{S}_{n,t,k}|) \le -t$ for $n$ large enough, as required.
\end{proof}

Moreover, the following holds by a similar argument.
Note that it can be verified in the case $\tau > 2$, corresponding to 
$\lfloor\kappa/\tau\rfloor=\lfloor1+1/(\tau-1)\rfloor=1$, that 
$\iota(\tau,\kappa+\eps) = \kappa+\eps-\tau(\kappa+\eps-\tau)$ provided that $\eps>0$ 
is small enough.

\begin{proposition}\label{prop:exp,mediumt,upper,1}
Suppose $p = p(n)$ satisfies $0 < p < 1$ and $n p \to \infty$ as $n \to \infty$,  and $\eps>0$ is a small enough constant.
Suppose $t=t(n)$ and $k = k(n)$ satisfy as $n \to \infty$ that $t \sim \tau \log_b np$ and 
$k \sim (\kappa+\eps) \log_b np$, where $\tau,\kappa>0$ satisfy $\tau > 2$ and $\iota(\tau,\kappa)=0$.
Then $\Exp(|\mathcal{S}_{n,t,k}|) \le \exp((1+o(1))\iota(\tau,\kappa+\eps)(\log np)^2/\log b)$.
\end{proposition}

\begin{proof}
Since $\eps>0$ can be chosen small and $n$ taken large enough, we may assume based on $\tau > 2$ and $\iota(\tau,\kappa)=0$ that $t \ge k/2$.
Following the proof of Lemma~\ref{lem:exp,larget,upper}, and since $\log b = \Theta(p)$, we obtain
\begin{align*}
\log \Exp(|\mathcal{S}_{n,t,k}|)
& \le k \log \left(\frac{en}{k}\right) + k\log 3 - t(k-t)\log b
 = k\log n p - t(k-t) \log b + o(k\log n p) \\
& \sim (\kappa+\eps-\tau(\kappa+\eps-\tau))\frac{(\log np)^2}{\log b}
 = \iota(\tau,\kappa+\eps)\frac{(\log np)^2}{\log b}. \qedhere
\end{align*}
\end{proof}

For the next first moment upper bounds, we require a bound on the number $\mathcal{SP}_{t,k}$ of set partitions 
of $[k]$ with block sizes at most $t$.  An easy application of the 
saddle-point method from analytic combinatorics, cf.~Flajolet and Sedgewick~\cite{FlSe09}, suffices. 
The proof of the following can be found in the appendix.

\begin{proposition}\label{prop:SPtk}
If $t \le \log k$, then for $k$ large enough
\begin{align*}
\mathcal{SP}_{t,k} \le \exp\left(k\log k-\frac{k}{t} \log k - k\log t + 3k\right).
\end{align*}
\end{proposition}

Note that the size of a largest part in a randomly chosen set partition of $[k]$ is 
$(1+o(1))\log k$, cf.~\cite{FlSe09}. Thus, if $t > \log k$, we instead appeal to a 
general asymptotic bound for set partitions, cf.~\cite[Proposition~VIII.3]{FlSe09}, 
which implies that
\begin{align}\label{eqn:SPtk}
\mathcal{SP}_{t,k} \le \mathcal{SP}_{k,k}
 \le (1+ o(1))\frac{k!}{(\log k)^k}
 = \exp(k \log k - k\log\log k - k +o(k)).
\end{align}
The following two bounds are consequences of these set partition estimates.

\begin{lemma}\label{lem:exp,smallt,upper,1}
Suppose $p = p(n)$ satisfies $0 < p < 1$.
Suppose $t=t(n)$ and $k = k(n)$ satisfy  as $n \to \infty$ that $t\le \log \log_b np$ and
\begin{align*}
k \ge 2\log_b n + t -2\log_b t- \frac{2\log_b \log_b np}t + \frac{10}{\log b}.
\end{align*}
Then $\Exp(|\mathcal{S}_{n,t,k}|) \le \exp(-k)$ for $n$ large enough.
\end{lemma}

\begin{lemma}\label{lem:exp,smallt,upper,2}
Suppose $p = p(n)$ satisfies $0 < p < 1$.
Suppose $t=t(n)$ and $k = k(n)$ satisfy  as $n \to \infty$ that $t = o(\log_b n)$ and
\begin{align*}
k \ge 2\log_b n + t -2\log_b \log \log_b n + \frac{3}{\log b}.
\end{align*}
Then $\Exp(|\mathcal{S}_{n,t,k}|) \le \exp(-k)$ for $n$ large enough.
\end{lemma}

\begin{proof}[Proof of Lemma~\ref{lem:exp,smallt,upper,1}]
Let us define $\hat{k} = t\lfloor k/t\rfloor$.
Any $t$-component $k$-set induces a set partition of $[k]$ into blocks of size at most $t$, 
such that there is no edge between vertices of two different blocks.
The total number of non-edges among the blocks is minimised by having the least number $\hat{k}/t+1$ of blocks with all but one of the blocks having size exactly $t$.
Such a partition has at least $\binom{\hat{k}/t}{2} t^2 + \hat{k}(k-\hat{k})$ non-edges.

We have $t \le \log \log_b np \le \log k$.
(To see this, note that it holds for $t=1$, then use monotonicity in $t$ of the bound on $k$.)
Thus, using Proposition~\ref{prop:SPtk} and $\binom{x}{y} \le (e x/y)^y$, we have for $n$ large enough
\begin{align*}
\Exp(|\mathcal{S}_{n,t,k}|)
& \le \binom{n}{k} q^{\binom{\hat{k}/t}{2} t^2 + \hat{k}(k-\hat{k})} \mathcal{SP}_{t,k}\\
& \le \left(\frac{en}{k}\right)^k q^{\frac{\hat{k}(\hat{k}-t)}{2} + 
\hat{k}(k-\hat{k})} \exp\left(k\log k-\frac{k}{t} \log k - k\log t + 3k\right).
\end{align*}
Taking the logarithm, dividing by $k/2$, substituting $\log k \ge \log \log_b np$, and simplifying, we get
\begin{align*}
&\frac{2\log \Exp(|\mathcal{S}_{n,t,k}|)}{k} \\
& \le 2\log n -\left(k-t+\frac{(k-\hat{k})(t-(k-\hat{k}))}{k}\right)\log b-
\frac{2\log \log_b np}{t}-2\log t+8 \\
& \le 2\log n -(k-t)\log b-\frac{2\log \log_b np}{t}-2\log t+8.
\end{align*}
The second inequality above follows from the fact that $0 \le k-\hat{k} < t$.
Substituting the assumed lower bound on $k$, we obtain the desired result.
\end{proof}

\begin{proof}[Proof of Lemma~\ref{lem:exp,smallt,upper,2}]
We follow the previous proof, but we substitute the general bound of~\eqref{eqn:SPtk} 
instead of Proposition~\ref{prop:SPtk}.  If $\hat{k} = t\lfloor k/t\rfloor$, this yields
\begin{align*}
\Exp(|\mathcal{S}_{n,t,k}|)
& \le \left(\frac{en}{k}\right)^k q^{\frac{\hat{k}(\hat{k}-t)}{2} + \hat{k}(k-\hat{k})} 
\exp(k \log k - k\log\log k - k +o(k)),
\end{align*}
and then (since $\log k \ge \log \log_b n$)
\begin{align*}
&\frac{2\log \Exp(|\mathcal{S}_{n,t,k}|)}{k} \\
& \le 2\log n -\left(k-t+\frac{(k-\hat{k})(t-(k-\hat{k}))}{k}\right)\log b-2\log \log \log_b n + o(1) \\
& \le 2\log n -(k-t)\log b-2\log \log \log_b n+1.
\end{align*}
Then substitution of the assumed lower bound on $k$ yields the result.
\end{proof}

By a similar argument, we see moreover that the following is true.

\begin{proposition}\label{prop:exp,mediumt,upper,2}
Suppose $0 < p < 1$ and $\eps>0$ are fixed.
Suppose $t=t(n)$ and $k = k(n)$ satisfy as $n \to \infty$ that $t \sim \tau \log_b n$ 
and $k \sim (\kappa+\eps) \log_b n$, where $\tau,\kappa>0$ satisfy $\tau \le 2$ and
$\iota(\tau,\kappa) =0$.
Then $\Exp(|\mathcal{S}_{n,t,k}|) \le \exp((1+o(1))\iota(\tau,\kappa+\eps)(\log n)^2/\log b)$.
\end{proposition}

\begin{proof}
Following the last proof, if $\hat{k} = t\lfloor k/t\rfloor$, then we obtain
\begin{align*}
\frac{2\log \Exp(|\mathcal{S}_{n,t,k}|)}{k}
& \le (2+o(1))\log n -\left(k-t+\frac{(k-\hat{k})(t-(k-\hat{k}))}{k}\right)\log b\\
& \sim -\left(\kappa+\eps-\tau-2 + \frac{\left(\kappa+\eps - \tau\left\lfloor \frac{\kappa+\eps}{\tau}\right\rfloor\right)\left(\tau-\kappa-\eps + \tau\left\lfloor \frac{\kappa+\eps}{\tau}\right\rfloor\right)}{\kappa+\eps}\right)\log n,
\end{align*}
whereupon we have derived
\begin{align*}
\log \Exp(|\mathcal{S}_{n,t,k}|)
& \le (1+o(1))\iota(\tau,\kappa+\eps)\frac{(\log n)^2}{\log b}.\qedhere
\end{align*}
\end{proof}

\subsection{Lower bounds on $\Exp(|\mathcal{S}_{n,t,k}|)$}

We now establish lower bounds for $\Exp(|\mathcal{S}_{n,t,k}|)$.  
First we remind the reader of the following.

\begin{proposition}[Erd\H{o}s and R\'enyi~\cite{ErRe59}]
\label{prop:conn}
For any $0 < p < 1$ and positive integer $t$ satisfying $tp \ge 2\log t$, there exists 
$\eta = \eta(t,p) > 2/3$ such that $\Pr(\G{t}{p}\text{ is connected}) \ge \eta$ for all 
$t$ sufficiently large.
\end{proposition}

\begin{lemma}\label{lem:exp,larget,lower}
Suppose $p = p(n)$ satisfies $0 < p < 1$ and $n p \to \infty$ as $n \to \infty$.
Suppose $t=t(n)$ and $k = k(n)$ satisfy that $t,k \to \infty$ as $n \to \infty$.
Furthermore assume $t > k/2$ (so that $1 \le k/t < 2$),
\begin{align*}
k \le t + \frac{k}{t}\log_b \frac{np}{pt + \log np} -\frac{1}{t}\log_b\frac{4}{\eta}- \frac{1}{\log b},
\end{align*}
where $\eta=\eta(t,p)$ is as in Proposition~\ref{prop:conn}.
Then $\Exp(|\mathcal{S}_{n,t,k}|) \ge \exp(t)$ for $n$ large enough.
\end{lemma}

\begin{proof}
For this, we count $t$-component $k$-sets formed by the disjoint union of a 
connected $t$-set and a connected $(k-t)$-set.  Given a set of $k$ vertices, we construct 
such a set by taking an arbitrary vertex subset with $t$ vertices, forming an arbitrary connected 
graph on those $t$ vertices, and forming an arbitrary graph on the remaining $k-t$ vertices. 
The choices of graph formed on the two parts can be made independently.  We have not 
double-counted any graph by this construction.  It follows by Proposition~\ref{prop:conn} that
\begin{align*}
\Exp(|\mathcal{S}_{n,t,k}|)
& \ge \binom{n}{k} \binom{k}{t} q^{t(k-t)}\eta.
\end{align*}
Since $\binom{x}{y} \ge (x/y)^y$, it then follows that
\begin{align*}
\log \Exp(|\mathcal{S}_{n,t,k}|)
 & \ge k \log \frac{n}{k} + t \log \frac{k}{t} - t(k-t) \log b +\log\eta\\
 & \ge k\log np - k\log pk - t(k-t) \log b +\log\eta.
\end{align*}
The conditions on $k$ and $t$ imply both that $pk < 2(pt + \log np)$ and
\begin{align*}
t(k-t)\log b
& \le k\log np - k\log(pt + \log np) - \log\frac{4}{\eta} - t\\
& \le k\log np - k\log pk + 2\log 2 - \log\frac{4}{\eta} - t. 
\end{align*}
Therefore,
\begin{align*}
\frac{\log \Exp(|\mathcal{S}_{n,t,k}|)}t
& \ge \frac{\log\eta-2\log2}{t} + \frac{1}{t}\log\frac{4}{\eta} + 1
 = 1,
\end{align*}
as required.
\end{proof}

Moreover, a similar argument shows that the following holds.
Recall that in the case $\tau > 2$, corresponding to $\lfloor\kappa/\tau\rfloor=1$, we have $\iota(\tau,\kappa-\eps) = \kappa-\eps-\tau(\kappa-\eps-\tau)$ if $\eps$ is small enough.

\begin{proposition}\label{prop:exp,mediumt,lower,1}
Suppose $0 < p < 1$ is fixed and $\eps>0$ is a small enough constant.
Suppose $t=t(n)$ and $k = k(n)$ satisfy as $n \to \infty$ that $t \sim \tau \log_b n$ and $k \sim (\kappa-\eps) \log_b n$, where $\tau,\kappa>0$ satisfy $\tau > 2$ and
$\iota(\tau,\kappa) =0$.
Then $\Exp(|\mathcal{S}_{n,t,k}|) \ge \exp((1+o(1))\iota(\tau,\kappa-\eps)(\log n)^2/\log b)$.
\end{proposition}

\begin{proof}
For $p$ fixed, the conditions of Proposition~\ref{prop:conn} are easily satisfied.
Moreover, based on $\tau > 2$ and $\iota(\tau,\kappa)=0$, we may assume $t > k/2$ for $n$ large enough.
Following the proof of Lemma~\ref{lem:exp,larget,lower}, we obtain
\begin{align*}
\log \Exp(|\mathcal{S}_{n,t,k}|)
& \ge k \log \frac{n}{k} + t \log \frac{k}{t} - t(k-t) \log b +\log\eta\\
& = k\log n - t(k-t) \log b + o((\log n)^2) \\
& \sim (\kappa-\eps-\tau(\kappa-\eps-\tau))\frac{(\log n)^2}{\log b}
 = \iota(\tau,\kappa-\eps)\frac{(\log n)^2}{\log b}. \qedhere
\end{align*}
\end{proof}

For the next lower bound, we need an expression for the number $\mathcal{EP}_{t,k}$ of set partitions of $[k]$ having the maximum number of parts of size exactly $t$.
For this, define $\hat{k} = t\lfloor k/t\rfloor$.
We can then write
\begin{align*}
\mathcal{EP}_{t,k} = \frac{k!}{(\hat{k}/t)!(t!)^{\hat{k}/t}(k-\hat{k})!}.
\end{align*}
By Stirling's approximation, we obtain that
\begin{align}
\mathcal{EP}_{t,k}
& \ge \frac{\sqrt{2\pi}k^{k+1/2}e^{-k}}{(e(\hat{k}/t)^{\hat{k}/t+1/2}e^{-\hat{k}/t})(et^{t+1/2}e^{-t})^{\hat{k}/t}(e(k-\hat{k})^{k-\hat{k}+1/2}e^{-(k-\hat{k})})}\nonumber\\
& = \Omega(1)\frac{k^{k+1/2}t^{\hat{k}/(2t)+1/2}}{\hat{k}^{\hat{k}/t+1/2}t^{\hat{k}}(k-\hat{k})^{k-\hat{k}+1/2}}
= \Omega(1)\frac{k^{k}t^{\hat{k}/(2t)}}{\hat{k}^{\hat{k}/t}t^{\hat{k}}(k-\hat{k})^{k-\hat{k}}}\frac{k^{1/2}t^{1/2}}{\hat{k}^{1/2}(k-\hat{k})^{1/2}}\nonumber\\
& \ge \Omega(1)\frac{k^{k}\sqrt{t}^{\hat{k}/t}}{k^{k/t}t^k}.
\label{eqn:EPtk}
\end{align}

\begin{lemma}\label{lem:exp,smallt,lower}
Suppose $p = p(n)$ satisfies $0 < p < 1$ and $n p \to \infty$ as $n \to \infty$.
Suppose $t=t(n)$ and $k = k(n)$ satisfy as $n \to \infty$ that $t\to\infty$, 
$t \le 2\log_b np$, $k \ge \log_b np$ and 
\begin{align*}
k \le 2\log_b n + t - \frac{t^2}{4\log_b np} - 2\log_b t - \frac{2\log_b \log_b np}{t} 
+\frac{2\log(\eta \sqrt{t}/3)}{3\log np}-\frac{1}{\log b},
\end{align*}
where $\eta = \eta(t,p)$ is as in Proposition~\ref{prop:conn}.
Then $\Exp(|\mathcal{S}_{n,t,k}|) \ge \exp(k)$ for $n$ large enough.
\end{lemma}

\begin{proof}
For this lower bound, it suffices to count $t$-component $k$-sets formed based on the disjoint 
union of $\hat{k}/t$ connected $t$-sets.  We construct such sets by taking set partitions 
of $[k]$ of the form counted by $\mathcal{EP}_{t,k}$, and independently forming an arbitrary 
connected graph on each block of size $t$ (and an arbitrary graph on the remainder block, 
if necessary).  Note that the number of non-edges for such a $t$-component $k$-set is 
bounded below by $\binom{\hat{k}/t}{2} t^2 + \hat{k}(k-\hat{k})$ (where $\hat{k} = t\lfloor k/t\rfloor$).  Each set constructed 
in this way is a $t$-component $k$-set and no set is double-counted.  It follows from 
Proposition~\ref{prop:conn} and~\eqref{eqn:EPtk} that
\begin{align*}
&\Exp(|\mathcal{S}_{n,t,k}|)
\ge \binom{n}{k} q^{\binom{\hat{k}/t}{2} t^2 + \hat{k}(k-\hat{k})}\eta^{\hat{k}/t}\mathcal{EP}_{t,k}\\
& \ge \binom{n}{k} q^{\frac{\hat{k}(\hat{k}-t)}{2} 
+ \hat{k}(k-\hat{k})}\eta^{\hat{k}/t}\exp\left(k\log k-\frac{k}{t} \log k - k\log t 
+ \frac{\hat{k}}{t}\log \sqrt{t} + o(k) \right).
\end{align*}
The assumed upper bound on $k$ implies that $k\le 3\log_b np$.
Now, using $\binom{x}{y} \ge (x/y)^y$, taking the logarithm, dividing by $k/2$, substituting 
$\log k \le \log \log_b np + \log 3$, we obtain for $n$ large enough
\begin{align*}
&\frac{2\log\Exp(|\mathcal{S}_{n,t,k}|)}{k}\\
& \ge 2\log n-\left(k-t+\frac{(k-\hat{k})(t-(k-\hat{k}))}{k}\right)\log b-\frac{2\log k}{t}
-2\log t+\frac{2\log (\eta \sqrt{t})}{k}+o(1)\\
& \ge 2\log n-\left(k-t+\frac{t^2}{4\log_b np}\right)\log b-2\log t-\frac{2\log \log_b np}{t}
+\frac{2\log (\eta \sqrt{t}/3)}{k}-1.
\end{align*}
The result follows upon substitution of the assumed upper bound on $k$
(and $k\le3\log_b np$).
\end{proof}

By a similar argument, we see moreover that the following holds.

\begin{proposition}\label{prop:exp,mediumt,lower,2}
Suppose $0 < p < 1$ and $\eps>0$ are fixed.
Suppose $t=t(n)$ and $k = k(n)$ satisfy as $n \to \infty$ that $t \sim \tau \log_b n$ and 
$k \sim (\kappa-\eps) \log_b n$, where $\tau,\kappa>0$ satisfy $\tau \le 2$ and $\iota(\tau,\kappa) =0$.
Then $\Exp(|\mathcal{S}_{n,t,k}|) \ge \exp((1+o(1))\iota(\tau,\kappa-\eps)(\log n)^2/\log b)$.
\end{proposition}

\begin{proof}
For $p$ fixed, the conditions of Proposition~\ref{prop:conn} are satisfied.
Following the proof of Lemma~\ref{lem:exp,smallt,lower},
\begin{align*}
&\frac{2\log \Exp(|\mathcal{S}_{n,t,k}|)}{k}\\
& \ge 2\log n-\left(k-t+\frac{(k-\hat{k})(t-(k-\hat{k}))}{k}\right)\log b-\frac{2\log k}{t}-2\log t+\frac{2\log (\eta \sqrt{t})}{k}+o(1)\\
& = (2+o(1))\log n -\left(k-t+\frac{(k-\hat{k})(t-(k-\hat{k}))}{k}\right)\log b\\
& \sim -\left(\kappa-\eps-\tau-2 + \frac{\left(\kappa-\eps - \tau\left\lfloor \frac{\kappa-\eps}{\tau}\right\rfloor\right)\left(\tau-\kappa+\eps + \tau\left\lfloor \frac{\kappa-\eps}{\tau}\right\rfloor\right)}{\kappa-\eps}\right)\log n.
\end{align*}
Then, as desired, we have derived
\begin{align*}
\log \Exp(|\mathcal{S}_{n,t,k}|)
& \ge (1+o(1))\iota(\tau,\kappa-\eps)\frac{(\log n)^2}{\log b}.\qedhere
\end{align*}
\end{proof}


\section{The threshold: $t=\Theta(\log n)$}\label{sec:chi}

This section is devoted to carrying out a second moment estimate to prove the following lemma.

\begin{lemma}\label{lem:chi,medium}
Suppose $0 < p < 1$ is fixed and $\eps>0$ is a small enough constant.
Suppose $t=t(n) \sim \tau \log_b n$ as $n \to \infty$ for some $\tau > 0$, and let $\kappa=\kappa(\tau)$ be the unique positive real guaranteed by Lemma~\ref{lem:iota}.
If $k = k(n) \sim (\kappa-\eps) \log_b n$ as $n\to\infty$,
then
$\Pr(\alphac{t}(\G{n}{p}) < k) \le \exp(-n^2/(\log n)^5)$.
\end{lemma}

Let us first see how this lemma implies our main theorem.
This same approach was core to determining the asymptotic behaviour of $\chi(\G{n}{p})$ in~\cite{Bol88}.

\begin{proof}[Proof of Theorem~\ref{thm:chi,medium}]
Let $\eps>0$ be some arbitrary small constant.
It follows from Propositions~\ref{prop:exp,mediumt}\ref{prop:exp,mediumt,upper,part} and 
Lemma~\ref{lem:careful}\ref{prop:careful,minus} that
\begin{align*}
\Pr\left(\chic{t}(\G{n}{p}) \le \frac{n}{(\kappa+\eps) \log_b n} \right)
& \le \Pr\left(\alphac{t}(\G{n}{p}) \ge (\kappa+\eps) \log_b n\right) \\
& \le \Exp(|\mathcal{S}_{n,t,k}|)
=  \exp(-\Omega((\log n)^2))
\end{align*}
(where $\mathcal{S}_{n,t,k}$ is the collection of $t$-component $k$-sets in $\G{n}{p}$); 
thus $\chic{t}(\G{n}{p}) \ge n/((\kappa+\eps) \log_b n)$ a.a.s.
The remainder of the proof is devoted to obtaining a closely matching upper bound.

For this, set $k = (\kappa-\eps/2)\log_b n$.
Let $\mathcal{A}_n$ denote the set of graphs $G$ on $[n]$ such that $\alphac{t}(G[S]) \ge k$ 
for all $S \subseteq [n]$ with $|S| \ge n/(\log n)^2$.  Then, by Lemma~\ref{lem:chi,medium}, 
assuming $\eps$ is small enough,
\begin{align*}
\Pr\left(\G{n}{p} \notin \mathcal{A}_n\right)
& \le 2^n \Pr\left(\alphac{t}\left(\G{\lceil n/(\log n)^2 \rceil}{p}\right) < k\right) 
\le \exp\left(O(n)-\Omega\left(n^2/(\log n)^9\right)\right) \to 0
\end{align*}
as $n \to \infty$.  Therefore, $\G{n}{p} \in \mathcal{A}_n$ a.a.s.

But for a graph $G$ in $\mathcal{A}_n$ the following procedure yields a colouring as desired. 
Let $S' = [n]$.  While $|S'| \ge n/(\log n)^2$, form a colour class from an arbitrary $t$-component 
$k$-subset $T$ of $S'$ and let $S' = S'\setminus T$.  At the end of these iterations, 
$|S'| < n/(\log n)^2$ and we may just assign each vertex of $S'$ to its own colour class. 
The resulting partition is a $t$-component colouring of $\G{n}{p}$ and
the total number of colours used is less than $n/((\kappa-\eps/2) \log_b n) 
+ n/(\log n)^2 \le n/((\kappa - \eps)\log_b n)$ for large enough $n$.
As $\eps >0$ was chosen arbitrarily small, this completes the proof.
\end{proof}

\begin{proof}[Proof of Lemma~\ref{lem:chi,medium}]
Throughout the proof, we always assume a choice of $\eps>0$ that is small enough for our 
purposes --- for the application of Lemma~\ref{lem:careful} we certainly need at least that
$\eps < \min\left\{\tau,\kappa-\tau\left\lfloor \frac{\kappa}{\tau}\right\rfloor\right\}$.
Then from Proposition~\ref{prop:exp,mediumt}\ref{prop:exp,mediumt,lower,part} we have as $n\to\infty$ that
\begin{align}
\Exp(|\mathcal{S}_{n,t,k}|)
& \ge \exp\left((1+o(1))\iota(\tau,\kappa-\eps)\frac{(\log n)^2}{\log b}\right)
\ge \exp\left((1+o(1))\eps\frac{(\log n)^2}{\log b}\right).
\label{eqn:exp}
\end{align}
We use Janson's Inequality (Theorem~2.18(ii) in~\cite{JLR00}):
\begin{align}
\Pr(\alphac{t}(\G{n}{p}) < k) 
= \Pr(|\mathcal{S}_{n,t,k}| = 0) 
\le \exp \left( - \frac{ (\Exp(|\mathcal{S}_{n,t,k}|))^2}{\Exp(|\mathcal{S}_{n,t,k}|) + \Delta }\right), 
\label{eqn:Janson,chi}
\end{align}
where
\[
\Delta = \sum_{A, B \subseteq [n], 1 < |A \cap B| < k} \Pr(A, B \in \mathcal{S}_{n,t,k}).
\]
We split $\Delta$ into separate sums according to the size $\ell$ of $A \cap B$. 
In particular, let $p(k,\ell)$ be the probability that two $k$-subsets of $[n]$ that
overlap on exactly $\ell$ vertices are both in $\mathcal{S}_{n,t,k}$. 
Thus 
\[ 
\Delta = \sum_{2\le \ell < k} f(\ell), \qquad \text{ where }\qquad 
f(\ell) = \binom{n}{k} \binom{k}{\ell} \binom{n - k}{k-\ell} p(k,\ell). 
\]
Set $\ell_1 = \lambda_1\log_b n$ and $\ell_2 = \lambda_2\log_b n$, for some 
$0\le\lambda_1\le \lambda_2 \le \kappa$ which are chosen to satisfy the inequalities~\eqref{eqn:lambda1},~\eqref{eqn:lambda2,1} and~\eqref{eqn:lambda2,2} below.
Now we write $\Delta = \Delta_1 + \Delta_2 + \Delta_3$ where $\ell_1,\ell_2$ determine the 
ranges of the sums into which we decompose $\Delta$:
\begin{align*}
\Delta_1 &=
 \sum_{2\le \ell < \ell_1} f(\ell ), \qquad
 \Delta_2 = \sum_{\ell_1 \le \ell < \ell_2} f(\ell), \qquad
 \Delta_3 = \sum_{\ell_2 \le \ell < k} f(\ell).
\end{align*}
It suffices to show that $\Delta_i = O((\log n)^5/n^2) (\Exp(|\mathcal{S}_{n,t,k}|))^2$ for each $i\in\{1,2,3\}$ for the result to follow from~\eqref{eqn:Janson,chi}.
To bound each $\Delta_i$ we consider two arbitrary $k$-subsets $A$ and $B$ of $[n]$ that overlap on exactly $\ell$ vertices,
i.e.~$|A\cap B| = \ell$. 
Moreover, we write
\begin{align*}
p(k,\ell) = \Pr(A,B \in \mathcal{S}_{n,t,k}) = \Pr( A \in \mathcal{S}_{n,t,k} \given B \in \mathcal{S}_{n,t,k}) \Pr(B \in \mathcal{S}_{n,t,k})
\end{align*}
and focus on bounding the conditional factor.
We remark here that, although rounding is indeed quite important to the form of this result, we shall several times in optimisation procedures below take the liberty of discarding floor and ceiling symbols, wherever this causes no confusion.

\paragraph{Bounding $\Delta_1$.}\label{p:delta1_begin}

The property of having component order at most $t$ is monotone decreasing, so 
the conditional probability that $A \in \mathcal{S}_{n,t,k}$ is maximised when $E[A\cap B] = \emptyset$. 
Thus
\begin{align*}
\Pr( A \in \mathcal{S}_{n,t,k} \given B \in \mathcal{S}_{n,t,k})
& \le \Pr( A \in \mathcal{S}_{n,t,k} \given E[A\cap B] = \emptyset) \\
& \le \frac{\Pr(A \in \mathcal{S}_{n,t,k})}{\Pr(E[A\cap B] = \emptyset)} = 
b^{\binom{\ell}{2}} \Pr(A \in \mathcal{S}_{n,t,k}),
\end{align*}
implying that $p(k,\ell) \le b^{\binom{\ell}{2}}(\Pr(A \in \mathcal{S}_{n,t,k}))^2$.  

We have though that for $n$ large enough
\begin{align*}
\frac{\binom{k}{\ell} \binom{n-k}{k-\ell}}{\binom{n}{k}}
\le 2\frac{\binom{k}{\ell} n^{k-\ell} / (k-\ell )!}{{n^k/k!}} 
\le 2\left(\frac{k^2}{n}\right)^\ell.
\end{align*}
Thus
\begin{align*}
\Delta_1
& \le \left(\binom{n}{k}\Pr(A \in \mathcal{S}_{n,t,k})\right)^2 \sum_{2 \le \ell < \ell_1} 2\left(\frac{k^2}{n}\right)^\ell b^{\binom{\ell}{2}}
= (\Exp(|\mathcal{S}_{n,t,k}|))^2 \sum_{2 \le \ell < \ell_1} s_\ell,
\end{align*}
where
\begin{align}\label{eq:def_sell}
s_{\ell}:= 2\left(\frac{k^2}{n}\right)^\ell b^{\binom{\ell}{2}}.
\end{align}
We now show that the summation $\sum s_\ell$ is $o(1)$.
To this end, note that $s_{\ell+1}/s_\ell = k^2 b^\ell/n$ and so the sequence $\{s_\ell\}$ 
is convex in $\ell$.
So $s_\ell$ is maximised over $\ell\in\{2,\dots,\ell_1\}$ at either $\ell = 2$ or $\ell = \ell_1$.
We have that $s_2 = 2bk^4/n^2 = \Theta((\log n)^4/n^2)$, but
\begin{align*}
s_{\ell_1}
& \le 2\left(\frac{k^2}{n}n^{\lambda_1/2}\right)^{\lambda_1\log_b n} = \exp(-\Omega((\log n)^2)),
\end{align*}
provided that $\lambda_1$ is chosen so that
\begin{align}
0 < \lambda_1 < 2.
\label{eqn:lambda1}
\end{align}
Therefore, with this choice,
\begin{align*}
\Delta_1
 \le \ell_1 s_2 (\Exp(|\mathcal{S}_{n,t,k}|))^2 = O\left(\frac{(\log n)^5}{n^2}\right) (\Exp(|\mathcal{S}_{n,t,k}|))^2.
\end{align*}

\paragraph{Bounding $\Delta_3$.}

In this case, we implicitly use the assumption that $\lambda_2 > \tau$, but as we shall see this is automatic from the requirement~\eqref{eqn:lambda2,1} below.
Given that $B \in \mathcal{S}_{n,t,k}$, let us lower bound the number of non-edges accounted for by $A\setminus B$ with the event $A\in \mathcal{S}_{n,t,k}$. (So we count those non-edges induced by $A\setminus B$ plus those induced between $A\setminus B$ and $A\cap B$.)
In this event, we know that each vertex of $A\setminus B$ has maximum degree less than $t$ in $A$. The overall contribution of such vertices to the number of non-edges will be smallest if each neighbourhood is strictly contained in $A\cap B$. We conclude that the number of non-edges accounted for is at least $(k-\ell)(\ell-t+1)+\binom{k-\ell}{2} \ge (k-\ell)(\ell-t)+(k-\ell)^2/2$.
From this, and also using a crude bound for the number of set partitions of $A$, we get
\begin{align*}
\Pr( A \in \mathcal{S}_{n,t,k} \given B \in \mathcal{S}_{n,t,k})
& \le k^kq^{\frac12(k-\ell)(\ell+k-2t)}.
\end{align*}
Thus, since $\binom{k}{\ell} \binom{n-k}{k-\ell}
\le (kn)^{k-\ell}$, we have
\begin{align*}
\Delta_3
 \le \Exp(|\mathcal{S}_{n,t,k}|) \sum_{\ell_2 \le \ell < k} s_{\ell},
 \text{ where }s_{\ell}:= k^k\left(kn\cdot b^{-\frac12(\ell+k)+t}\right)^{k-\ell}.
\end{align*}
We now show that the summation $\sum s_\ell$ is $o(\Exp(|\mathcal{S}_{n,t,k}|))$. 
To this end, note that
\begin{align*}
\frac{s_{\ell+1}}{s_\ell}
 = \frac{b^{\ell-t+1/2}}{kn}
\end{align*}
and so the sequence $\{s_\ell\}$ is convex in $\ell$. So $s_\ell$ is maximised over $\ell\in\{\ell_2,\dots,k-1\}$ at either $\ell = \ell_2$ or $\ell = k-1$.
We have that $ks_{k-1} = k^{k+2}nb^{-k+t+1/2} = k^{O(k)}$. On the other end,
\begin{align*}
s_{\ell_2}
& \le k^k\left(kn\cdot b^{-\frac12(\ell_2+k)+t}\right)^{k-\ell_2}
= \exp\left(\left(1-\frac12(\lambda_2+\kappa-\eps)+\tau+o(1)\right)(\kappa-\lambda_2)\frac{(\log n)^2}{\log b}\right).
\end{align*}
Therefore, comparing with~\eqref{eqn:exp}, we may conclude that $\sum s_\ell \le ks_{\ell_2} = o(\Exp(|\mathcal{S}_{n,t,k}|))$ provided we choose
\begin{align}
\lambda_2
 \ge 2+2\tau-\kappa+\eps-\hat{\eps}
\label{eqn:lambda2,1}
\end{align}
for any $0 < \hat{\eps} < \eps$ satisfying
$\hat{\eps}(\kappa-\tau-1-(\eps-\hat{\eps})/2) < \iota(\tau,\kappa-\eps)$.
Since $\kappa \le \tau+2$, this automatically implies $\lambda_2>\kappa$.
Moreover, with any choice satisfying~\eqref{eqn:lambda2,1},
we may conclude that $\Delta_3 = O((\log n)^5/n^2) (\Exp(|\mathcal{S}_{n,t,k}|))^2$.
Note that $\iota(\tau,\kappa-\eps) \ge \eps$ by Lemma~\ref{lem:careful}, guaranteeing a choice for $\hat{\eps}$.
The reason for the restriction $\hat{\eps} < \eps$ is that, if we are in the case of Lemma~\ref{lem:careful}\ref{prop:careful,divide} and choose both $\hat{\eps} = \eps$ and $\lambda_2 = 2+2\tau-\kappa$, then $\hat{\eps}(\kappa-\tau-1-(\eps-\hat{\eps})/2) = \eps = \iota(\tau,\kappa-\eps)$ so that $s_{\ell_2}$ cannot be guaranteed to be smaller than the expression in~\eqref{eqn:exp}.
Since $\kappa > \tau+1$, we can also guarantee that the choice of $\lambda_2$ satisfies
\begin{align}
\lambda_2<\kappa,
\label{eqn:lambda2,2}
\end{align}
provided $\eps$ is small enough.

\paragraph{Bounding $\Delta_2$.}

In first bounding $\Delta_1$ and $\Delta_3$, we have derived appropriate conditions on the choice of $\lambda_1$ and $\lambda_2$, in inequalities~\eqref{eqn:lambda1},~\eqref{eqn:lambda2,1} and~\eqref{eqn:lambda2,2}.
Before beginning our analysis of $\Delta_2$, we note that $\kappa > 2\tau$ for all $0 < \tau < 2$; otherwise, $\tau \lfloor \kappa /\tau \rfloor < 2$ and it follows from $\iota(\tau,\kappa)=0$ that $\kappa > 2(\tau+4)/3$ which is greater than $\tau+2$ for $0 < \tau < 2$, a contradiction to Lemma~\ref{lem:iota}.
We may therefore assume that $\tau \ge 2$, or else the summation $\Delta_2$ can be made empty with a small enough choice of $\eps$ and a choice of $\lambda_1$ close enough to $2$.

Note that every $t$-component $k$-set induces a bipartition so that one part has at least $k-t$ vertices, the other has at least $t/2$ vertices, and there are no edges between the two parts. (To build such a partition, we form one of the parts by including just the largest component, unless it has at most $t/2$ vertices, in which case we add just the second largest component to the part, unless the resulting set has at most $t/2$ vertices, and so on.) For each such bipartition corresponding to $A$ being a $t$-component set, there is a corresponding bipartition of $A\setminus B$ (one part possibly being empty). We can thus estimate $\Pr( A \in \mathcal{S}_{n,t,k} \given B \in \mathcal{S}_{n,t,k})$ by conditioning on the bipartition of $A\setminus B$, and consider its extensions to bipartitions of $A$. Taking into account the non-edges between the parts, and by deeming the part of at least $k-t$ vertices to be composed of $i$ vertices from $A\setminus B$ and $j$ vertices from $A\cap B$, we obtain
\begin{align*}
\Pr( A \in \mathcal{S}_{n,t,k} \given B \in \mathcal{S}_{n,t,k}) 
& \le \max_{0\le i\le k-\ell} \binom{k-\ell}{i} \sum_{j=k-t-i}^{\min\{\ell,k-t/2-i\}} \binom{\ell}{j} q^{i(k-i-j)+(k-\ell-i)(i+j)-i(k-\ell-i)}  \\
& \le k^{O(k)} \max_{0\le i\le k-\ell} \sum_{j=k-t-i}^{\min\{\ell,k-t/2-i\}} q^{i(k-i)+j(k-\ell-2i)}.
\end{align*}
We break this maximisation in half with cases $i\le (k-\ell)/2$ and $i\ge (k-\ell)/2$, corresponding to different signs for $k-\ell-2i$.

In the lower half, the sum is maximised by minimising $j$, so
\begin{align*}
\sum_{j=k-t-i}^{\min\{\ell,k-t/2-i\}} q^{i(k-i)+j(k-\ell-2i)} 
& \le kq^{i(k-i)+(k-t-i)(k-\ell-2i)} 
 = kq^{i^2-(2(k-t)-\ell)i+(k-t)(k-\ell)}.
\end{align*}
Note that the convex quadratic in the exponent of this last expression is minimised at $i = k-t-\ell/2$. It can be checked that this value of $i$ is no larger than $(k-\ell)/2$, since $\tau > 1$; however, if $\ell > 2(k-t)$, then this value of $i$ is smaller than $0$, in which case the minimum of the quadratic is at $i=0$. We conclude that
\begin{align}
\max_{0\le i\le (k-\ell)/2} \sum_{j=k-t-i}^{\min\{\ell,k-t/2-i\}} q^{i(k-i)+j(k-\ell-2i)} 
& \le \left\{ \begin{array}{ll}
kq^{t(k-t)-\ell^2/4} & \text{if $\ell \le 2(k-t)$}\\
kq^{(k-t)(k-\ell)} & \mbox{otherwise}
\end{array} \right..
\label{eqn:lowersum}
\end{align}

In the upper half, the sum is maximised by maximising $j$. 
First consider when $k-t/2-i$ is the minimum in the upper delimiter for $j$, and so
\begin{align*}
\sum_{j=k-t-i}^{\min\{\ell,k-t/2-i\}} q^{i(k-i)+j(k-\ell-2i)}
 & \le kq^{i(k-i)+(k-t/2-i)(k-\ell-2i)}
 = kq^{i^2-(2(k-t/2)-\ell)i+(k-t/2)(k-\ell)}.
\end{align*}
Note the convex quadratic in the exponent of this last expression is minimised at $i = k-t/2-\ell/2$. It can be checked that this value of $i$ is no smaller than $(k-\ell)/2$, since $k \ge t$; however, if $\ell > t$, then this value of $i$ is larger than $k-\ell$, in which case the minimum of the quadratic is at $i=k-\ell$.
We conclude that
\begin{align}
\max_{(k-\ell)/2 \le i \le k-\ell} \sum_{j=k-t-i}^{\min\{\ell,k-t/2-i\}} q^{i(k-i)+j(k-\ell-2i)}
& \le \left\{ \begin{array}{ll}
kq^{\frac12t(k-t/2)-\ell^2/4} & \text{if $\ell \le t$}\\
kq^{\frac12t(k-\ell)} & \mbox{otherwise}
\end{array} \right..
\label{eqn:uppersum,first}
\end{align}
Otherwise $\ell\le k-t/2-i$ and so
in this case one concludes from a comparison of the extreme values of $i$, namely $i = (k-\ell)/2$ and $i = k-\ell-t/2$, that $\ell \le k-t$. This scenario is ruled out by a choice of $\lambda_1 > 2-\eps/2 > \kappa-\eps-\tau$ (using that $\tau\ge2$).

For the final stage of our estimate of $\Delta_2$, it will suffice to assume that $\ell \sim \lambda\log_b n$ for some $\lambda_1\le\lambda\le\lambda_2$. Since $\binom{k}{\ell} \binom{n-k}{k-\ell} \le k^{O(k)}n^{k-\ell}$, we can write
\begin{align}
\log \frac{f(\ell)}{\Exp(|\mathcal{S}_{n,t,k}|)} \le (1+o(1)) (\kappa-\eps-\lambda)\frac{(\log n)^2}{\log b}
+\log\Pr( A \in \mathcal{S}_{n,t,k} \given B \in \mathcal{S}_{n,t,k}),
\label{eqn:fell}
\end{align}
and shall show the expression is at most any fixed positive fraction of $(\log n)^2$ (and indeed could be negative) using~\eqref{eqn:lowersum} and~\eqref{eqn:uppersum,first}.

If we are in the first subcase of~\eqref{eqn:lowersum}, then $\lambda\le2(\kappa-\eps-\tau)$, and so we can conclude that
\begin{align*}
\log \frac{f(\ell)}{\Exp(|\mathcal{S}_{n,t,k}|)}
& \le (1+o(1)) \left(\kappa-\eps-\lambda-\tau(\kappa-\eps-\tau)+\frac14\lambda^2\right)\frac{(\log n)^2}{\log b}\\
&\sim \left(\lambda^2-4\lambda+4\eps(\tau-1)\right)\frac{(\log n)^2}{4\log b},
\end{align*}
where we used $\iota(\tau,\kappa)=0$ and $\kappa=\tau+\tau/(\tau-1)$. Consider the polynomial in $\lambda$ in brackets in the above expression. It has roots $2\pm2\sqrt{1-\eps(\tau-1)}$. So, since $\lambda_1$ is arbitrarily close to $2$ independently of $\eps$, the entire expression above is bounded above by any fixed fraction of $(\log n)^2$ provided
\begin{align*}
2+2\sqrt{1-\eps(\tau-1)} \ge 2(\kappa-\eps-\tau) = \frac{2\tau}{\tau-1}-2\eps.
\end{align*}
Since $\tau\ge2$, this inequality is guaranteed by a small enough choice of $\eps$.

If we are in the second subcase of~\eqref{eqn:lowersum}, then by~\eqref{eqn:fell}
\begin{align*}
\log \frac{f(\ell)}{\Exp(|\mathcal{S}_{n,t,k}|)}
& \le (1+o(1)) \left(\kappa-\eps-\lambda-(\kappa-\eps-\tau)(\kappa-\eps-\lambda)\right)\frac{(\log n)^2}{\log b}\\
& \sim \left((1-\kappa+\eps+\tau)(\kappa-\eps-\lambda)\right)\frac{(\log n)^2}{\log b}.
\end{align*}
which is at most any fixed fraction of $(\log n)^2$ with a small enough choice of $\eps$, since $\kappa>\tau+1$ and $\lambda_2 < \kappa$ (by~\eqref{eqn:lambda2,2}).

If we are in the first subcase of~\eqref{eqn:uppersum,first}, then $\lambda\le\tau$, and we deduce using~\eqref{eqn:fell} that
\begin{align*}
\log \frac{f(\ell)}{\Exp(|\mathcal{S}_{n,t,k}|)}
& \le (1+o(1)) \left(\kappa-\eps-\lambda-\frac12\tau\left(\kappa-\eps-\frac12\tau\right)+\frac14\lambda^2\right)\frac{(\log n)^2}{\log b} \\
& \sim \left(\left(1-\frac12\tau\right)\kappa-\eps-\lambda+\frac12\eps\tau+\frac14\tau^2+\frac14\lambda^2\right)\frac{(\log n)^2}{\log b}\\
& = \left(\lambda^2-4\lambda-\tau^2\left(1-\frac2{\tau-1}\right)+2\eps(\tau-2)\right)\frac{(\log n)^2}{4\log b}
\end{align*}
where in the last two lines we used $\iota(\tau,\kappa)=0$ and $\kappa=\tau^2/(\tau-1)$.
Consider the polynomial in $\lambda$ in brackets in the last line. It has roots
\begin{align*}
2\pm\sqrt{4+\tau^2\left(1-\frac2{\tau-1}\right)-2\eps(\tau-2)},
\end{align*}
and so the expression in the last line above is at most any fixed fraction of $(\log n)^2$ provided
\begin{align*}
2+\sqrt{4+\tau^2\left(1-\frac2{\tau-1}\right)} \ge \tau,
\end{align*}
since $\eps$ can be made arbitrarily small. This inequality holds by the fact that $\tau\ge2$.

If we are in the second subcase of~\eqref{eqn:uppersum,first}, then by~\eqref{eqn:fell}
\begin{align*}
\log \frac{f(\ell)}{\Exp(|\mathcal{S}_{n,t,k}|)}
& \le (1+o(1)) \left(\kappa-\eps-\lambda-\frac12\tau(\kappa-\eps-\lambda)\right)\frac{(\log n)^2}{\log b}\\
& \sim \left(1-\frac12\tau\right)(\kappa-\eps-\lambda)\frac{(\log n)^2}{\log b},
\end{align*}
which is at most any fixed fraction of $(\log n)^2$
with a small enough choice of $\eps$, since $\tau\ge2$ and $\lambda_2 < \kappa$.

We have succeeded in proving that $f(\ell) \le \Exp(|\mathcal{S}_{n,t,k}|)\cdot \exp(o((\log n)^2))$ if $\ell\sim\lambda\log_b n$ and $\lambda_1\le\lambda\le\lambda_2$. Since $\Exp(|\mathcal{S}_{n,t,k}|) \ge \exp(\Omega((\log n)^2))$ by~\eqref{eqn:exp}, this implies that $\Delta_2 = \sum_{\ell_1 \le \ell < \ell_2} f(\ell) \le O((\log n)^5/n^2) (\Exp(|\mathcal{S}_{n,t,k}|))^2$, as desired.

\medskip

Having obtained the desired estimates of $\Delta_1$, $\Delta_2$ and $\Delta_3$, we have completed the proof.
\end{proof}

\section{Constant-width concentration: $t\le \log\log_b np$}\label{sec:alpha}

In this section, we prove Theorem~\ref{thm:alpha}. We require a specialised Chernoff-type bound.
We define
\[
\Lambda^*(x) = \left\{ \begin{array}{ll}
\displaystyle x \log \frac{x}{p} + (1 - x) \log \frac{1 - x}{q} & \mbox{for $x\in [0, 1]$}\\
\infty & \mbox{otherwise}
\end{array} \right.,
\]
where $\Lambda^*(0) = \log b$ and $\Lambda^*(1) = \log (1/p)$.
This is the Fenchel--Legendre transform of the logarithmic moment generating function for 
the Bernoulli distribution with probability $p$.

\begin{lemma}[Lemma~3.3 of~\cite{KaMc10}]\label{lem:mixedbin}
Let $n_1$ and $n_2$ be positive integers, let $0<p<1$, and let $X$ and $Y$ be independent 
random variables with $X \sim \Bin(n_1,p)$ and $Y/2 \sim \Bin(n_2,p)$.
Note that $\Exp(X+Y)= (n_1+2 n_2)p$. Then for $0\leq x \leq p$
\[
\Pr(X+Y \leq (n_1+2 n_2)x) \leq \exp \left( - \frac12 (n_1+2 n_2) \Lambda^*(x)\right).
\]
\end{lemma}

\begin{proof}[Proof of Theorem~\ref{thm:alpha}]
Due to Proposition~\ref{prop:exp,smallt}\ref{prop:exp,smallt,upper,part}, 
this proof reduces to proving a lower bound on $\alphac{t}(\G{n}{p})$.  
Let us note that, with the choice
\begin{align*}
k \le 2\log_b n + t - 2\log_b t - \frac{2\log_b \log_b np}{t} -\frac{2}{\log b},
\end{align*}
Proposition~\ref{prop:exp,smallt}\ref{prop:exp,smallt,lower,part} implies
$\Exp(|\mathcal{S}_{n,t,k}|) \ge \exp(k)$
for $n$ large enough.

As in the course of the proof of Theorem~\ref{thm:chi,medium} (p.\ \pageref{eqn:Janson,chi})
we use Janson's Inequality. The setting here is similar and the proof naturally 
follows similar lines. We have
\begin{align}
\Pr(\alphac{t}(\G{n}{p}) < k) = \Pr(|\mathcal{S}_{n,t,k}| = 0) \le \exp \left( - \frac{ (\Exp(|\mathcal{S}_{n,t,k}|))^2}{\Exp(|\mathcal{S}_{n,t,k}|) + \Delta }\right), \label{eqn:Janson}
\end{align}
where
\[
\Delta = \sum_{A, B \subseteq [n], 1 < |A \cap B| < k} \Pr(A, B \in \mathcal{S}_{n,t,k})
\]
(and $\mathcal{S}_{n,t,k}$ is the collection of $t$-component $k$-sets in $\G{n}{p}$).
Recall that $p(k,\ell)$ denotes the probability that two $k$-subsets of $[n]$ that
overlap on exactly $\ell$ vertices are both in $\mathcal{S}_{n,t,k}$. Thus
\[ 
\Delta = \sum_{2\le \ell < k} f(\ell), 
\qquad\text{ where }\qquad 
f(\ell) = \binom{n}{k} \binom{k}{\ell} \binom{n - k}{k-\ell} p(k,\ell). 
\]
One difference from the proof of Theorem~\ref{thm:chi,medium} is that here we split $\Delta$ 
into only two sums: 
we set $\ell_1 = 2\log_b n - 6\log_b k$ and write $\Delta = \Delta_1 + \Delta_2$ where 
$\ell_1$ determines the split of the sum:
\begin{align*}
\Delta_1 &=
 \sum_{2\le \ell \le \ell_1} f(\ell ),\qquad \text{and} \qquad
 \Delta_2 = \sum_{\ell_1 < \ell < k} f(\ell).
\end{align*}
It suffices to show that $\Delta_i = o((\Exp(|\mathcal{S}_{n,t,k}|))^2)$ for each $i\in\{1,2\}$ 
for the result to follow from~\eqref{eqn:Janson}.
To bound each $\Delta_i$ we consider two arbitrary $k$-subsets $A$ and $B$ of $[n]$ that 
overlap on exactly $\ell$ vertices, i.e.~$|A\cap B| = \ell$, and estimate $p(k,\ell)$ by 
conditioning on the set $E[A\cap B]$ of edges induced by $A\cap B$. 
In order to bound $p(k,\ell)$, we focus on the conditional probability 
$\Pr( A \in \mathcal{S}_{n,t,k} \given B \in \mathcal{S}_{n,t,k})$.

It is worth noting the basic estimates, $k \le (2+o(1))\log_b n$ and $k-\ell_1 = O(\log\log n)$. Furthermore, we may safely assume that $k$ is chosen so that $k \ge \log_b n$. We also ignore some rounding below, where it is unimportant.

\paragraph{Bounding $\Delta_1$.}

Our bound on $\Delta_1$ follows the same argument as for $\Delta_1$ in the proof of Theorem~\ref{thm:chi,medium}, 
and only differs at the very end when replacing $\ell_1$ by its value. We refer the reader 
to the arguments on page~\pageref{p:delta1_begin} for more details. The convex sequence $(s_\ell)$ 
defined there in~\eqref{eq:def_sell} is such that $s_2 = 2bk^4/n^2$, and 
\begin{align*}
s_{\ell_1}
& \le 2\left(\frac{k^2}{n}b^{\ell_1/2}\right)^{\ell_1} = \frac{1}{k^{\ell_1}} = o(s_2).
\end{align*}
Therefore, by convexity (proved on page~\pageref{p:delta1_begin}),
\begin{align*}
\sum_{2\le \ell \le \ell_1} s_{\ell} \le \ell_1 s_2 = O\left(\frac{k^5}{n^2}\right) = o(1).
\end{align*}

\label{p:delta1_end}
\paragraph{Bounding $\Delta_2$.} 
Note that
\begin{align*}
\Pr(A \in \mathcal{S}_{n,t,k} \given B \in \mathcal{S}_{n,t,k})
 \le \Pr( \forall v\in A\setminus B, \deg_{A}(v) \le t ),
\end{align*}
where $\deg_S(v)$ denotes the number of neighbours of $v$ in $S$.  
It therefore follows that
\begin{align*}
\Pr(A \in \mathcal{S}_{n,t,k} \given B \in \mathcal{S}_{n,t,k})
& \le \Pr\left(\sum_{v \in A \setminus B} \deg_A(v) \le t(k-\ell)\right)\\
& = \Pr\left(\Bin(\ell(k-\ell),p)+ 2\Bin\left(\binom{k-\ell}{2},p\right) \le t(k-\ell) \right).
\end{align*}
We shall employ Lemma~\ref{lem:mixedbin} with $n_1=\ell(k-\ell)$, $n_2=\binom{k-\ell}{2}$, and $x=t/(k-1)$. Note that $n_1+2 n_2= (k-1)(k-\ell)$ and so $(n_1+2 n_2)x=t(k-\ell)$.
Since $x=o(p)$, it follows from Taylor expansion calculations found in the first paragraph of the appendix of~\cite{FKM14} that
\begin{align*}
\Lambda^*\left(\frac{t}{k-1} \right)
= \log b - (1+o(1))\frac{t}{k}\log \frac{pk}{t} = \log b - (1+o(1))\frac{t}{k}\log \log n.
\end{align*}
Hence we conclude by Lemma~\ref{lem:mixedbin} that
\begin{align*}
\Pr(A \in \mathcal{S}_{n,t,k} \given B \in \mathcal{S}_{n,t,k})
& \le \exp\left(-\frac12 (k-1)(k-\ell) \Lambda^*\left(\frac{t}{k-1} \right)\right)\\
& = \exp\left(-(1+o(1))\left(1-\frac{t}{k}\log_b \log n\right)(k-\ell) \log n\right)\\
& = \left(\frac{(\log n)^{t/2}}{n}\right)^{(1+o(1))(k-\ell)}.
\end{align*}
Since $\binom{k}{\ell}\binom{n-k}{k-\ell} \le (k n)^{k-\ell}$, $k \le (2+o(1))\log_b n$ and $k-\ell_1 = O(\log\log n)$, we obtain that
\begin{align*}
\Delta_2
& \le \binom{n}{k} \Pr(B \in \mathcal{S}_{n,t,k})\sum_{\ell_1<\ell<k}
 \left((\log n)^{1+t/2}\right)^{(1+o(1))(k-\ell)} \\
& = \Exp(|\mathcal{S}_{n,t,k}|)\cdot \exp(O(t(\log\log n)^2)).
\end{align*}
That this last expression is $o((\Exp(|\mathcal{S}_{n,t,k}|))^2)$ follows by noting that $\Exp(|\mathcal{S}_{n,t,k}|) = \exp(\Omega(\log n))$ and $t = O(\log\log n)$.

\medskip

We have appropriately bounded $\Delta_1$ and $\Delta_2$, concluding the proof.
\end{proof}


\section{Sparse random graphs}\label{sec:sparse}

We do not have a complete understanding of $\chic{t}(\G{n}{p})$ and $\alphac{t}(\G{n}{p})$ for $p\to0$ as $n\to \infty$.
Nonetheless, we can observe the phenomenon described at the beginning of the paper: in any partition of the vertices of $\G{n}{p}$ into asymptotically fewer than $\chi(\G{n}{p})$ parts, one of the parts must induce a subgraph having a large component, about as large as the average part size.
This follows directly from the next result.

\begin{theorem}\label{thm:chi}
Suppose $p = p(n)$ satisfies $0 < p < 1$ and $n p \to \infty$ as $n \to \infty$.  Then the following hold.
\begin{enumerate}
\itemsep0pt
\item\label{thm:chi,smallestt} If $t(n) = o(\log np)$, then $\chic{t}(\G{n}{p}) \sim n/(2 \log_b np)$ a.a.s.
\item\label{thm:chi,smallt} If $t(n) = o(\log_b np)$, then $(1-o(1))n/(4\log_b np) \le \chic{t}(\G{n}{p}) \le (1+o(1))n/(2\log_b np)$ a.a.s.
\item\label{thm:chi,mediumt} If $t(n) = \Theta(\log_b np)$ and $t(n) = o(n)$, then $\chic{t}(\G{n}{p}) = \Theta\left( n/\log_b np \right) = \Theta\left( n/t \right)$ a.a.s.
\item\label{thm:chi,larget,1} If $t(n) = \omega(\log_b np)$ and $t(n) = o(n)$, then $\chic{t}(\G{n}{p}) \sim n /t$ a.a.s.
\item\label{thm:chi,larget,2} If $t(n) \sim n/x$, where $x > 0$ is fixed and not integral, then $\chic{t}(\G{n}{p}) = \lceil x \rceil$ a.a.s.
\end{enumerate}
\end{theorem}

Proposition~\ref{prop:exp,mediumt,upper,1} immediately implies the following.

\begin{proposition}\label{prop:alpha}
Suppose $p = p(n)$ satisfies $0 < p < 1$ and $n p \to \infty$ as $n \to \infty$.
If $t(n) \sim \tau\log_b np$ for some $\tau> 2$, then 
$\alphac{t}(\G{n}{p}) \le (\tau+1+1/(\tau-1) + o(1))\log_b np$ a.a.s. 
\end{proposition}

Let us see how this upper bound on $\alphac{t}(\G{n}{p})$ is used to obtain Theorem~\ref{thm:chi}.

\begin{proof}[Proof of Theorem~\ref{thm:chi}]
The upper bounds of Theorem~\ref{thm:chi} follow from Proposition~\ref{prop:basic}, and previously mentioned results for $\chi(\G{n}{p})$.  For the lower bounds, we use that $\chic{t}(\G{n}{p}) \ge n/\alphac{t}(\G{n}{p})$, and apply Proposition~\ref{prop:alpha} with $\tau$ arbitrarily close to $2$ for~\ref{thm:chi,smallt}, $\tau$ fixed for~\ref{thm:chi,mediumt}, or $\tau$ arbitrarily large for~\ref{thm:chi,larget,1} and~\ref{thm:chi,larget,2}.
The case~\ref{thm:chi,smallestt} is implied by Theorem~1.3 of~\cite{KaMc10}.
\end{proof}

Note that in the setting of Theorem~1.3 of~\cite{KaMc10}, i.e.~colourings with bounded monochromatic average degree, the analogous threshold is $t = \Theta(\log np)$ which is asymptotically smaller than the $t = \Theta(p^{-1}\log np)$ threshold implicit in Theorem~\ref{thm:chi}.
We remark that Lemma~\ref{lem:exp,smallt,upper,2} does not suffice to completely narrow the gap in Theorem~\ref{thm:chi}\ref{thm:chi,smallt}.
Moreover, in the intermediate case~\ref{thm:chi,mediumt}, one might expect an analogue of Theorem~\ref{thm:chi,medium} to hold.
However, we leave these two problems to future study.

\section{Component Ramsey numbers}\label{sec:ramsey}

In this section, we consider the Ramsey-type numbers based on bounded sized components.
The next proof closely follows~\cite{Erd47}. A constant-factor improvement would be available here using the Lov\'asz Local Lemma, as in~\cite{Spe77}, but we expect that further improvements would be much more difficult to obtain.

\begin{proof}[Proof of Proposition~\ref{prop:ramseylower}]
For any $\delta>0$ and some large enough integer $k$, let
\begin{align*}
n = \left\lfloor \frac{1}{1+\delta}\frac{k}{3e}2^{\eps(1-\eps)k}\right\rfloor.
\end{align*}
Let $G$ be distributed as $\G{n}{1/2}$.
Given a subset $S\subseteq [n]$ of $k$ vertices, let $A_S$ be the event that $S$ is a $\lfloor(1-\eps)k\rfloor$-component set in $G$ or its complement.
By exactly the same arguments used to obtain~\eqref{eqn:star} in Lemma~\ref{lem:exp,larget,upper} (with $t=\lfloor(1-\eps)k\rfloor$), since $\eps < 1/2$, we see that
\begin{align*}\label{eqn:exp,larget,upper}
\Pr(A_S)
 \le 2\cdot 3^k \cdot 2^{-\eps(1-\eps)k^2}.
\end{align*}
So the probability that $A_S$ holds for some $S$ is at most
\begin{align*}
\sum_{S\subseteq [n],|S|= k}\Pr(A_S)
& \le 2 \binom{n}{k} 3^k \cdot 2^{-\eps(1-\eps)k^2}
\le 2 \left(\frac{en\cdot 3}{k} \cdot 2^{-\eps(1-\eps)k} \right)^k
 \le 2 (1+\delta)^{-k} < 1.
\end{align*}
Thus, for $k$ large enough, there exists a graph on $n$ vertices in which no $k$-subset is a $\lfloor(1-\eps)k\rfloor$-component set in the graph or its complement. We proved this for all $\delta>0$, so the result follows.
\end{proof}

We contrast Proposition~\ref{prop:ramseylower} with upper bounds of the following form. The first of these compares with Proposition~\ref{prop:ramseylower} when $\eps$ is near $1/2$, while the second of these when $\eps$ is near $0$. Both show that there is limited room for improvement in Proposition~\ref{prop:ramseylower}.
\begin{proposition}
As $k\to\infty$,
\begin{align*}
R^{\frac12(k+\log_2k-1)}(k) \le (1+o(1)) \sqrt{k}2^{\frac12(k-1)}.
\end{align*}
Fix $0 < c < 1$. Then, as $k\to\infty$,
\begin{align*}
R^{k-\frac{c}{1-c}\log_2k+1}(k) \le (1+o(1))k^{1/(1-c)}.
\end{align*}
\end{proposition}

\begin{proof}
These bounds follow directly from a K\H{o}v\'ari--S\'os--Tur\'an result, Lemma~2 in~\cite{CFS11}, which guarantees complete bipartite subgraphs in dense graphs. Specifically, the lemma states, ``If a graph on $n$ vertices has $\epsilon n^2$ edges and $t < \epsilon n$, then it contains the complete bipartite graph $K_{s,t}$ with $s = \epsilon^tn$.'' Note that complete bipartite graphs and their induced subgraphs have bounded components in the complement.  For the first bound, we apply the lemma, either to a given graph on $n$ vertices or to its complement, with $\epsilon =1/2$ and $t=\log_2n-\log_2\log_2n$ to obtain $K_{\log_2 n, \log_2n-\log_2\log_2n}$. For the second we use $\epsilon = 1/2$ and $t = c\log_2 n$ to obtain $K_{c\log_2n,n^{1-c}}$.
\end{proof}

\subsection*{Acknowledgements}

We thank Guus Regts and Jean-S\'ebastien Sereni for insightful discussions about Section~\ref{sec:ramsey}.

\bibliographystyle{abbrv}
\bibliography{compchi}

\begin{thebibliography}{10}

\bibitem{ADOV03}
N.~Alon, G.~Ding, B.~Oporowski, and D.~Vertigan.
\newblock Partitioning into graphs with only small components.
\newblock {\em J. Combin. Theory Ser. B}, 87(2):231--243, 2003.

\bibitem{BeSz07}
R.~Berke and T.~Szab{\'o}.
\newblock Relaxed two-coloring of cubic graphs.
\newblock {\em J. Combin. Theory Ser. B}, 97(4), 2007.

\bibitem{BeSz09}
R.~Berke and T.~Szab{\'o}.
\newblock Deciding relaxed two-colourability: a hardness jump.
\newblock {\em Combin. Probab. Comput.}, 18(1-2):53--81, 2009.

\bibitem{BFKLS11}
T.~Bohman, A.~Frieze, M.~Krivelevich, P.-S. Loh, and B.~Sudakov.
\newblock Ramsey games with giants.
\newblock {\em Random Structures Algorithms}, 38(1-2):1--32, 2011.

\bibitem{Bol88}
B.~Bollob{\'a}s.
\newblock The chromatic number of random graphs.
\newblock {\em Combinatorica}, 8(1):49--55, 1988.

\bibitem{Bol01}
B.~Bollob{\'a}s.
\newblock {\em Random {G}raphs}, volume~73 of {\em Cambridge Studies in
  Advanced Mathematics}.
\newblock Cambridge University Press, Cambridge, 2nd edition, 2001.

\bibitem{BoEr76}
B.~Bollob{\'a}s and P.~Erd{\H{o}}s.
\newblock Cliques in random graphs.
\newblock {\em Math. Proc. Cambridge Philos. Soc.}, 80(3):419--427, 1976.

\bibitem{BoTh95}
B.~Bollob{\'a}s and A.~Thomason.
\newblock Generalized chromatic numbers of random graphs.
\newblock {\em Random Structures Algorithms}, 6(2-3):353--356, 1995.

\bibitem{BoTh00}
B.~Bollob{\'a}s and A.~Thomason.
\newblock The structure of hereditary properties and colourings of random
  graphs.
\newblock {\em Combinatorica}, 20:173--202, 2000.

\bibitem{BJM07}
T.~Britton, S.~Janson, and A.~Martin-L{\"o}f.
\newblock Graphs with specified degree distributions, simple epidemics, and
  local vaccination strategies.
\newblock {\em Adv. in Appl. Probab.}, 39(4):922--948, 2007.

\bibitem{Coj13}
A.~Coja-Oghlan.
\newblock Upper-bounding the {$k$}-colorability threshold by counting covers.
\newblock {\em Electron. J. Combin.}, 20:Paper 32, 28, 2013.

\bibitem{CoVi13}
A.~Coja-Oghlan and D.~Vilenchik.
\newblock The chromatic number of random graphs for most average degrees.
\newblock {\em Int. Math. Res. Not.}, 2016(19):5801--5859, 2016.

\bibitem{CFS11}
D.~Conlon, J.~Fox, and B.~Sudakov.
\newblock Large almost monochromatic subsets in hypergraphs.
\newblock {\em Israel J. Math.}, 181:423--432, 2011.

\bibitem{EdFa01}
K.~Edwards and G.~Farr.
\newblock Fragmentability of graphs.
\newblock {\em J. Combin. Theory Ser. B}, 82(1):30--37, 2001.

\bibitem{EdFa05}
K.~Edwards and G.~Farr.
\newblock On monochromatic component size for improper colourings.
\newblock {\em Discrete Appl. Math.}, 148(1):89--105, 2005.

\bibitem{EdFa08}
K.~Edwards and G.~Farr.
\newblock Planarization and fragmentability of some classes of graphs.
\newblock {\em Discrete Math.}, 308(12):2396--2406, 2008.

\bibitem{Erd47}
P.~Erd{\"o}s.
\newblock Some remarks on the theory of graphs.
\newblock {\em Bull. Amer. Math. Soc.}, 53:292--294, 1947.

\bibitem{ErRe59}
P.~Erd{\H{o}}s and A.~R{\'e}nyi.
\newblock On random graphs. {I}.
\newblock {\em Publ. Math. Debrecen}, 6:290--297, 1959.

\bibitem{ErSz35}
P.~Erd{\"o}s and G.~Szekeres.
\newblock A combinatorial problem in geometry.
\newblock {\em Compositio Math.}, 2:463--470, 1935.

\bibitem{EsJo14}
L.~Esperet and G.~Joret.
\newblock Colouring planar graphs with three colours and no large monochromatic
  components.
\newblock {\em Combin. Probab. Comput.}, 23(4):551--570, 2014.

\bibitem{EsOc14+}
L.~Esperet and P.~Ochem.
\newblock Islands in graphs on surfaces.
\newblock {\em SIAM J. Discrete Math.}, 30(1):206--219, 2016.

\bibitem{FlSe09}
P.~Flajolet and R.~Sedgewick.
\newblock {\em Analytic combinatorics}.
\newblock Cambridge University Press, Cambridge, 2009.

\bibitem{FKM10}
N.~Fountoulakis, R.~J. Kang, and C.~McDiarmid.
\newblock The {$t$}-stability number of a random graph.
\newblock {\em Electron. J. Combin.}, 17(1):Research Paper 59, 29, 2010.

\bibitem{FKM14}
N.~Fountoulakis, R.~J. Kang, and C.~McDiarmid.
\newblock Largest sparse subgraphs of random graphs.
\newblock {\em European J. Combin.}, 35:232--244, 2014.

\bibitem{GrMc75}
G.~R. Grimmett and C.~McDiarmid.
\newblock On colouring random graphs.
\newblock {\em Math. Proc. Cambridge Philos. Soc.}, 77:313--324, 1975.

\bibitem{HPT08}
P.~Haxell, O.~Pikhurko, and A.~Thomason.
\newblock Maximum acyclic and fragmented sets in regular graphs.
\newblock {\em J. Graph Theory}, 57:149--156, 2008.

\bibitem{HST03}
P.~Haxell, T.~Szab{\'o}, and G.~Tardos.
\newblock Bounded size components---partitions and transversals.
\newblock {\em J. Combin. Theory Ser. B}, 88(2):281--297, 2003.

\bibitem{JLR00}
S.~Janson, T.~{\L}uczak, and A.~Rucinski.
\newblock {\em Random {G}raphs}.
\newblock Wiley-Interscience Series in Discrete Mathematics and Optimization.
  Wiley-Interscience, New York, 2000.

\bibitem{JaTh08}
S.~Janson and A.~Thomason.
\newblock Dismantling sparse random graphs.
\newblock {\em Combin. Probab. Comput.}, 17(2):259--264, 2008.

\bibitem{KaMc10}
R.~J. Kang and C.~McDiarmid.
\newblock The {$t$}-improper chromatic number of random graphs.
\newblock {\em Combin. Probab. Comput.}, 19(1):87--98, 2010.

\bibitem{KaMc15}
R.~J. Kang and C.~McDiarmid.
\newblock Colouring random graphs.
\newblock In {\em Topics in chromatic graph theory}, volume 156 of {\em
  Encyclopedia Math. Appl.}, pages 199--229. Cambridge Univ. Press, Cambridge,
  2015.

\bibitem{Kaw09}
K.~Kawarabayashi.
\newblock A weakening of the odd {H}adwiger's conjecture.
\newblock {\em Combin. Probab. Comput.}, 17(6):815--821, 2008.

\bibitem{KaMo07}
K.~Kawarabayashi and B.~Mohar.
\newblock A relaxed {H}adwiger's conjecture for list colorings.
\newblock {\em J. Combin. Theory Ser. B}, 97(4):647--651, 2007.

\bibitem{KMRV97}
J.~M. Kleinberg, R.~Motwani, P.~Raghavan, and S.~Venkatasubramanian.
\newblock Storage management for evolving databases.
\newblock In {\em FOCS}, pages 353--362. IEEE Computer Society, 1997.

\bibitem{LMST08}
N.~Linial, J.~Matou{\v{s}}ek, O.~Sheffet, and G.~Tardos.
\newblock Graph colouring with no large monochromatic components.
\newblock {\em Combin. Probab. Comput.}, 17(4):577--589, 2008.

\bibitem{LiOu15+}
C.-H. Liu and S.~Oum.
\newblock Partitioning {$H$}-minor free graphs into three subgraphs with no
  large components.
\newblock {\em ArXiv e-prints}, Mar. 2015.

\bibitem{Luc91a}
T.~{\L}uczak.
\newblock The chromatic number of random graphs.
\newblock {\em Combinatorica}, 11(1):45--54, 1991.

\bibitem{MaPr08}
J.~Matou{\v{s}}ek and A.~P{\v{r}}{\'{\i}}v{\v{e}}tiv{\'y}.
\newblock Large monochromatic components in two-colored grids.
\newblock {\em SIAM J. Discrete Math.}, 22(1):295--311, 2008.

\bibitem{Mat70}
D.~W. Matula.
\newblock On the complete subgraphs of a random graph.
\newblock In {\em Proceedings of the 2nd Chapel Hill Conference on
  Combinatorial Mathematics and its Applications (Chapel Hill, N. C., 1970)},
  pages 356--369, 1970.

\bibitem{Mat72}
D.~W. Matula.
\newblock The employee party problem.
\newblock {\em Notices AMS}, 19(2):A--382, 1972.

\bibitem{Mat87}
D.~W. Matula.
\newblock Expose-and-merge exploration and the chromatic number of a random
  graph.
\newblock {\em Combinatorica}, 7(3):275--284, 1987.

\bibitem{MaKu90}
D.~W. Matula and L.~Ku{\v{c}}era.
\newblock An expose-and-merge algorithm and the chromatic number of a random
  graph.
\newblock In {\em Random {G}raphs '87 ({P}ozna\'n, 1987)}, pages 175--187.
  Wiley, Chichester, 1990.

\bibitem{McD90}
C.~J.~H. McDiarmid.
\newblock On the chromatic number of random graphs.
\newblock {\em Random Structures and Algorithms}, 1(4):435--442, 1990.

\bibitem{Rah14+}
M.~Rahman.
\newblock Percolation with small clusters on random graphs.
\newblock {\em Graphs Combin.}, 32(3):1167--1185, 2016.

\bibitem{Sch92}
E.~R. Scheinerman.
\newblock Generalized chromatic numbers of random graphs.
\newblock {\em SIAM J. Discrete Math.}, 5(1):74--80, 1992.

\bibitem{Spe77}
J.~Spencer.
\newblock Asymptotic lower bounds for {R}amsey functions.
\newblock {\em Discrete Math.}, 20(1):69--76, 1977/78.

\bibitem{SST10}
R.~Sp{\"o}hel, A.~Steger, and H.~Thomas.
\newblock Coloring the edges of a random graph without a monochromatic giant
  component.
\newblock {\em Electron. J. Combin.}, 17(1):Research Paper 133, 7, 2010.

\end{thebibliography}

\appendix

\section{Proofs of auxiliary technical results}

\begin{proof}[Proof of Lemma~\ref{lem:iota}]
To show that the function $\kappa$ is well-defined, fix $\tau$, $\kappa>0$ satisfying $\iota(\tau,\kappa)=0$ and write $\lfloor \kappa / \tau \rfloor = i$. Then the implicit equation is equivalent to
\[\kappa = \frac{\tau^2 i (i+1)}{2i \tau -2}.\]
Note that if $2i\tau-2=0$, then for $\iota(\tau,\kappa)=0$ to hold it must be that $(\kappa-1)(\kappa-1-\tau)=\kappa(\kappa-\tau-2)$ and so $\tau=-1$, contradicting our assumption on $\tau$.
Now, for $\lfloor \kappa/\tau \rfloor =i$ to hold, we must also have 
\begin{equation}
0\le \frac {\tau i(i+1)}{2i\tau -2} - i <1.
\end{equation}
It follows from this that $i\in (\frac 2 \tau, 1 + \frac 2 \tau ]$. There is precisely one integer in this interval, and so at most one solution to  $\iota(\tau,\kappa)=0$. One also verifies easily, by taking the above expression for $\kappa$ and $i=\lfloor 1 +\frac 2 \tau \rfloor$, that $\iota(\tau,\kappa)=0$ is indeed satisfied, and so there is exactly one solution. We conclude that $\kappa$ is defined by
\begin{equation}\label{eq:kappa_value}
	\kappa = \frac {\tau^2 i (i+1)}{2i\tau -2} \qquad \text {where } \qquad i = \left\lfloor 1 + \frac 2 \tau \right\rfloor.
\end{equation}

\begin{itemize}
\item[\ref{itm:iota,c}]
On each interval of the form $[\frac 2{j},\frac 2 {j-1})$, over which $i$ is invariant (and equals $j$), it is routine to check that $\kappa$ is a positive, continuous, increasing, convex function of $\tau$. The continuity on $(0,\infty)$ follows from the fact that, for every $i\ge 1$, 
\[\lim_{\tau \uparrow 2/i} \frac{\tau^2 (i+1)(i+2)}{2(i+1)\tau -2} = 2 + \frac 2 i = \lim_{\tau \downarrow 2/i} \frac{\tau^2 i (i+1)}{2i\tau -2}.\]
\item[\ref{itm:iota,d}]
 The second part follows readily from the formula for $\kappa$ in~\eqref{eq:kappa_value} together with part~\ref{itm:iota,c}. \qedhere
\end{itemize}
\end{proof}

\begin{proof}[Proof of Lemma~\ref{lem:careful}]
First note, for parts~\ref{prop:careful,minus} and~\ref{prop:careful,notdivide}, that it is routine to check that $\tau\left\lfloor \frac{\kappa}{\tau}\right\rfloor-1 > 1$.
\begin{enumerate}
\item
In this case, observe that
$\left\lfloor \frac{\kappa+\eps}{\tau}\right\rfloor = \left\lfloor\frac{\kappa}{\tau}\right\rfloor$. Using $\iota(\tau,\kappa)=0$, we write
\begin{align*}
2\iota(\tau,\kappa+\eps)
& = 0 + \eps\left(\kappa-\tau\left\lfloor\frac{\kappa}{\tau}\right\rfloor-\tau\right) + \eps\left(\kappa-\tau\left\lfloor \frac{\kappa}{\tau}\right\rfloor\right) + \eps^2
-\eps(\kappa-\tau-2)-\eps\kappa-\eps^2\\
& = -\eps\left(2\tau\left\lfloor \frac{\kappa}{\tau}\right\rfloor-2\right).
\end{align*}
\item
First observe that in this case
$\left\lfloor \frac{\kappa-\eps}{\tau}\right\rfloor = \left\lfloor\frac{\kappa}{\tau}\right\rfloor - 1 = \frac{\kappa}{\tau} - 1$.
Using this and the assumption $\iota(\tau,\kappa)=0$, we can write
$0=2\iota(\tau,\kappa) =-\kappa(\kappa-\tau-2)$ and
\begin{align*}
2\iota(\tau,\kappa-\eps)
& = (\tau-\eps)(-\eps)-(\kappa-\eps)(\kappa-\eps-\tau-2)\\
& = -\eps(\tau-\eps) +0+\eps(\kappa-\tau-2)+\eps k-\eps^2
 = 2\eps(\kappa-\tau-1).
\end{align*}
The equality follows from checking that $\kappa = \tau+2$ if $\tau | \kappa$.
\item
Observe in this case that
$\left\lfloor \frac{\kappa-\eps}{\tau}\right\rfloor = \left\lfloor\frac{\kappa}{\tau}\right\rfloor$. Using $\iota(\tau,\kappa)=0$, we write
\begin{align*}
2\iota(\tau,\kappa-\eps)
& = 0 - \eps\left(\kappa-\tau\left\lfloor\frac{\kappa}{\tau}\right\rfloor-\tau\right) - \eps\left(\kappa-\tau\left\lfloor \frac{\kappa}{\tau}\right\rfloor\right) + \eps^2
+\eps(\kappa-\tau-2)+\eps\kappa-\eps^2\\
& = \eps\left(2\tau\left\lfloor \frac{\kappa}{\tau}\right\rfloor-2\right). \qedhere
\end{align*}
\end{enumerate}
\end{proof}

\begin{proof}[Proof of Proposition~\ref{prop:SPtk}]
Recall that $\mathcal{SP}_{t,k}$ is the number of set partitions of $[k]$ with blocks of size at most $t$.
Following Note~VIII.12 of~\cite{FlSe09}, observe that $\mathcal{SP}_{t,k}$ is bounded by the product of $k!$ and the $z^k$ coefficients of the following exponential generating function:
\begin{align*}
SP_t(z) \equiv \exp\left(\sum_{i=1}^t \frac{z^i}{i!}\right).
\end{align*}
We have as $k\to \infty$ (cf.~Flajolet and Sedgewick~\cite[Corollary~VIII.2]{FlSe09})
\begin{align*}
[z^k]SP_t(z) \sim \frac{1}{\sqrt{2\pi\lambda}}\frac{SP_t(r)}{r^k}
\qquad\text{where}\qquad
\lambda = \left(r\frac{d}{dr}\right)^2 \sum_{i=1}^t \frac{r^i}{i!}
\end{align*}
and $r$ is given implicitly by the saddle-point equation
\begin{align*}
r\frac{d}{dr}\sum_{i=1}^t \frac{r^i}{i!} = k.
\end{align*}

We need to perform a few routine estimates.
First, we obviously have
\begin{align*}
\lambda
& = r^2 \frac{d^2}{dr^2}\left(\sum_{i=1}^t\frac{r^i}{i!}\right) + r \frac{d}{dr}\left(\sum_{i=1}^t\frac{r^i}{i!}\right)
 \ge k.
\end{align*}
Next, the implicit formula for $r$ is
\begin{align*}
k
& = r \sum_{i=1}^t\frac{r^{i-1}}{(i-1)!}
= \sum_{i=0}^{t-1}\frac{r^{i+1}}{i!}.
\end{align*}
So clearly
\begin{align*}
\frac{r^t}{(t-1)!} \le k \le e^k
\qquad\text{ and }\qquad
\log k \le r \le k^{1/t} ((t-1)!)^{1/t}.
\end{align*}
Now, since $t \le \log k \le r$, we see that the maximum of $r^{i+1}/i!$ in the range $i \in \{0,\dots,t-1\}$ is at $i = t-1$.  Thus we have
\begin{align*}
k \le \frac{r^t}{(t-2)!}
\qquad\text{ and }\qquad
k^{1/t} ((t-2)!)^{1/t} \le r.
\end{align*}
Therefore, we also obtain, using Stirling's approximation,
\begin{align*}
r^k
& \ge k^{k/t} ((t-2)!)^{k/t} \ge \exp\left(\frac{k}{t} \log k + k\left(\frac{t-2}{t}\log(t-2) -1\right)\right)\\
& \ge \exp\left(\frac{k}{t} \log k + k\log t - 1.9k\right).
\end{align*}
Furthermore,
\begin{align*}
SP_t(r)
& = \exp\left(\sum_{i=1}^t \frac{r^i}{i!}\right)
 \le \exp\left(\sum_{i=1}^t \frac{r^i}{(i-1)!}\right)
 = e^k
\end{align*}
Substituting these inequalities, we obtain
\begin{align*}
[z^k]SP_t(z)
& \le (1 + o(1))\frac{1}{\sqrt{2\pi k}}\exp\left(k -\frac{k}{t} \log k - k\log t + 1.9k\right).
\end{align*}
The result follows from an application of Stirling's approximation to $k!$ 
and a choice of $k$ large enough.
\end{proof}

\end{document}